\documentclass[a4paper,11pt]{article}
\usepackage[margin=1in]{geometry}
\usepackage{amsfonts,amsmath,amssymb,mathrsfs,amsthm,amscd,enumitem,multicol,multienum}
\usepackage{graphicx,multirow}
\usepackage{algorithm}
\usepackage{tabularx}
\usepackage{chemfig}
\usepackage[version=4]{mhchem}
\usepackage{bm}
\usepackage{latexsym}
\usepackage[utf8]{inputenc}
\usepackage[colorlinks=true, linkcolor=black, anchorcolor=black, citecolor=black, filecolor=black, urlcolor=black]{hyperref}
\usepackage{subcaption,float}

\newtheorem{thm}{Theorem}[section]
\newtheorem{lemma}[thm]{Lemma}

\numberwithin{equation}{section}

\renewcommand{\epsilon}{\varepsilon}

\begin{document}
	\begin{center}
		{\large Corrector estimates and numerical simulations of a system of diffusion-reaction-dissolution-precipitation model in a porous medium}

		\bigskip
		
		\bigskip
	\end{center}
	\begin{minipage}[h]{0.45\columnwidth}
		\begin{center}
			Nibedita Ghosh\\
			{\small Department of Mathematics, \\IIT Kharagpur}\\
			{\small WB 721302, India}\\
			{\small e-mail: nghosh.iitkgp@gmail.com}
		\end{center}
	\end{minipage}\hspace{1cm}
	\begin{minipage}[h]{0.45\columnwidth}
		\begin{center}
			Hari Shankar Mahato\\
			{\small Department of Mathematics, \\IIT Kharagpur}\\
			{\small WB 721302, India}\\
			{\small e-mail: hsmahato@maths.iitkgp.ac.in}
		\end{center}
	\end{minipage}

	\bigskip
	\hrule
	
	\bigskip
	
	\textbf{Abstract.} A system of diffusion-reaction equations coupled with a dissolution-precipitation model is discussed. We start by introducing a microscale model together with its homogenized version. In the present paper, we first derive the corrector result to justify the obtained theoretical results. Furthermore, we perform the numerical computations to compare the outcome of the effective model with the original heterogeneous microscale model.


	\bigskip
	\hrule
	
	\section{Introduction}
    Modeling crystal dissolution and precipitation through a porous medium is a topic of interest for a wide range of science fields, including chemical and tissue engineering. Examples include oil reservoir flow, groundwater flow and concrete carbonation, see for instance \cite{knabner1986free,knabner1995analysis,meier2007two,peter2008different,van1996crystal,van2004crystal} and references therein. More recently, modeling of transportation of chemical species received an increasing attention, cf. \cite{rubin1983transport,willis1987transport,peter2007scalings,peter2009multiscale}. Moreover, a porous medium is the union of the pore space and the solid parts. So we can view it in two different scales: one is the microscale which depicts the heterogeneities inside the medium but is not suited for numerical simulations, another one is the macroscale which describes the global behavior of the medium and is numerically efficient. The microscopic model serves as the starting point for the analysis. We then perform an upscaling based on two-scale convergence\cite{allaire1992homogenization,allairetwo,lukkassen2002two} and boundary unfolding operator \cite{cioranescu2006periodic,cioranescu2012periodic} in order to derive the macroscopic model. We restrict our analysis to the case of periodic homogenization only. However, in this paper, we are mainly focused on performing numerical simulations supported by the corrector result to investigate how well the upscaled equations approximate the original microscopic model.
	
	The main motivation of homogenization comes from the numerical point of view. The microscale model describes the physical and chemical phenomena at the pore scale. However, its numerical calculations seem very difficult, if not impossible. The numerical derivation of such a system will lead to a complicated analysis as the size of the step length should be chosen so small that it can catch the micro heterogeneities. That will result in enormous time consumption by the computer and a huge computational cost. Further, in real-world situations where numerous parameters are involved, the numerical computations do not seem to fit well. Therefore we require a homogeneous averaged model to perform simulations at a reasonable computational cost and at a convenient time. So basically, homogenization is a limiting procedure from a mathematical point of view. Here we look for a function $u_0$ such that $u_\epsilon\rightarrow u_0$ as $\epsilon\rightarrow 0$. Naturally, some questions arise such as
	\begin{itemize}
		\item What is the guaranty of existing such $u_0$?
		\item If it exists, in which sense the `\textit{limit}' is taken and what is the function space for the limit function?
		\item What can we say about its uniqueness?
		\item Does $u_0$ solve some limit boundary value problem?
		\item Are then the coefficients of the limit problem constant?
		\item Is $u_0$ is a good approximation of $u_\epsilon$?
		\item Finally, Do the upscaled model better suited for numerical simulations?
	\end{itemize} 
    We already answered the first five questions in our previous work \cite{ghoshmahato}. However, this paper is devoted to address the rest questions by conducting simulations and finding error bounds. 
	
	To be more specific, we now describe our main physical assumptions. We consider a pore-scale model for reactive flows. The void space is occupied by a fluid that contains two mobile species of different diffusion coefficients.  Reactions are happening at the pore space and produce an immobile species which precipitate on the grain boundary. The reverse reaction of dissolution also occurs. Several articles are exists on the derivation of corrector estimates for different classes of systems \cite{lukkassen2002two,cioranescu1999introduction,cioranescu2002periodic,cioranescu2008periodic,cioranescu2006periodic,donato2016periodic,muntean2013corrector,cioranescu2012periodic}. Numerical results in this direction can be found in \cite{van2009crystal,mahato2017numerical} and references therein. The main difficulty to derive the corrector result is to deal with the perforated porous medium.
	
   This paper is organized as follows: In Section 2, We discuss the periodic setting of the domain and the microscale model is introduced. In section 3, we first propose the technical assumptions needed for analysis. Later, we state and prove our main result on the corrector estimate. The main ingredients of the proof include integral estimates for oscillation functions and energy bounds. In section 4, we present and compare the numerical results for both the microscopic equations and the effective equations.

	\section{The model}
    Let $\Omega\subset \mathbb{R}^{n_{|n\ge 2}}$ be a periodic bounded domain such that $\Omega:=\Omega^{p}\cup\Omega^{s}$ and $\Bar{\Omega}^{s}\cap \Omega^{p}=\phi$, where $\Omega^p$ and $\Omega^s$ denotes the pore space and the union of the disconnected solid parts, respectively. $\partial\Omega$ and $\Gamma_\epsilon^*$ are the outer boundary of the domain and the union of boundaries of the solid parts. The unit representative cell $Y := (0,1)^n\subset \mathbb{R}^{n}$ is the union of a solid part $Y^s$ with boundary $\Gamma$ and the pore part $Y^p$ such that $Y=Y^{s}\cup Y^{p}$, $\bar{Y}^{s}\subset Y$ and $\bar{Y}^{s}\cap\bar{Y}^{p}=\Gamma$. The shifted set $Y^p_k$ is defined by $Y^p_k:=Y^p+\sum_{j=1}^{n}e_jk_j$, for $k=(k_1,k_2,\cdots, k_n)\in\mathbb{Z}^n$, where $e_j$ is the $j$th unit vector.
	The union of all shifted subsets $Y^p_k$ multiplied by $\epsilon$ (and confined within $\Omega$) defined the perforated porous medium $\Omega_\epsilon^p:=\cup_{k\in \mathbb{Z}^{n}}\left\{\varepsilon Y^{p}_{k}:\varepsilon Y^{p}_{k}\subset \Omega\right\}$. Similarly, $\Omega_\epsilon^s$ and $\Gamma_\epsilon^*$ denote the union of the shifted subsets $Y_k^s$ and $\Gamma_k$. The boundary of the pore space  $\partial \Omega _{\varepsilon}^{p}:=\partial \Omega \cup \Gamma^*_{\varepsilon}$. 
		\begin{figure}[h!]
		\begin{center}
			\includegraphics[width=12cm, height=5.8cm]{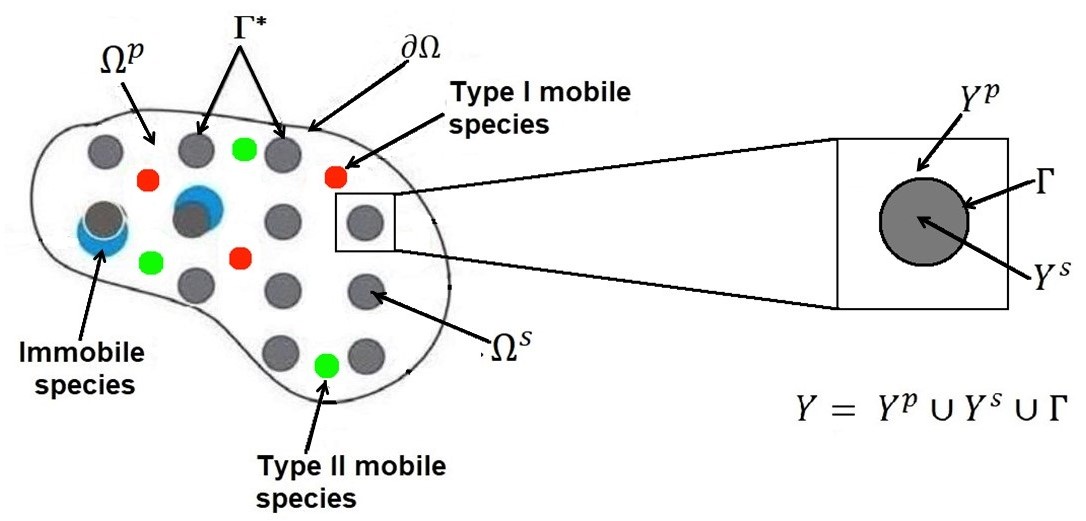}
			\caption{Mobile species in $\Omega^p$ with crystal dissolution and precipitation on $\Gamma^*$.}
		\end{center}
	\end{figure}
	We make the following geometric assumptions: $(a)$ solid parts do not touch each other. They are distributed periodically in the porous medium, $(b)$ solid parts do not touch the boundary of $Y$, $(c)$ solid parts do not touch the outer boundary $\partial\Omega$. For $T>0, S:=[0,T)$ denotes the time interval. We define the volume elements in $\Omega$ and $Y$ as $dx$ and $dy$ and the surface elements on $\Gamma_\epsilon^*, \Gamma$ by $d\sigma_x, d\sigma_y$, respectively. The characteristic function of $\Omega_\epsilon^p$ in $\Omega$ is given by
	\begin{align*}
		\chi_\epsilon(x)&\;\;=\;\;
		\left\{
		\begin{aligned}
			& 1 && \textnormal{if } x\in \Omega_\epsilon^p,\\
			& 0 && \textnormal{if } x\in \Omega\setminus\Omega_{\varepsilon}^{p}.
		\end{aligned}\right.
	\end{align*}
 According to our consideration, two mobile species $M_1$ and $M_2$ are present in $\Omega_\epsilon^p$. They are reacting reversibly and forming an immobile species $M_{12}$ which is accumulating on the interface between the pore space and the solid parts. The reaction amongst the mobile and immobile species is given by  
 \begin{equation}\label{eq2.5}
	M_1+M_2\leftrightarrow M_{12} \qquad \text{ on }\quad \Gamma_\varepsilon^*.
 \end{equation}
  We choose the forward reaction rate term as the \textit{Langmuir Isotherm} and the dissolution rate term is taken as $R_{d}(w_{\varepsilon})\in k_{d}\psi(w_{\varepsilon})$, cf. \cite{knabner1995analysis,van1996crystal,van2004crystal}, where
  \begin{align}\label{1.15}
  	\psi (c)&\;\;=\;\;
  	\left\{
  	\begin{aligned}
  		&\{0\} && \textnormal{if }c< 0,\\
  		&[0,1]&& \textnormal{if }c=0,\\
  		&\{1\} && \textnormal{if }c>0.
  	\end{aligned}\right.
  \end{align}
  We consider the same model as is described in \cite{ghoshmahato}. We denote the concentrations of $M_1, M_2$ and $M_{12}$ by $u_\epsilon, v_\epsilon$ and $w_\epsilon$, respectively. All these unknowns are dimensionless. Then the mass-balance equations for $M_1, M_2$ and $M_{12}$ are given by
	\begin{subequations}
		\begin{align}
			\frac{\partial u_\epsilon}{\partial {t}} + \nabla.(-{\bar{D}_1}\nabla u_\epsilon)&=0 \; \text{ in } S\times \Omega_\epsilon^p,	\label{eqn:M11}\\
			-{\bar{D}_1}\nabla u_\epsilon.\vec{n}&=0 \; \text{ on } S \times \partial\Omega, \label{eqn:M12}\\
			-{\bar{D}_1}\nabla u_\epsilon.\vec{n}&=\epsilon{\frac{\partial w_\epsilon}{\partial {t}}} \; \text{ on } S \times \Gamma_\epsilon^* ,\label{eqn:M13}\\
			u_\epsilon(0,x)&=u_{I\epsilon}(x) \; \text{ in } \Omega_\epsilon^p,\label{eqn:M14} 
		\end{align}	
		\begin{align}
			\frac{\partial v_\epsilon}{\partial {t}} + \nabla.(-{\bar{D}_2}\nabla v_\epsilon)&=0 \; \text{ in } S \times \Omega_\epsilon^p,	\label{eqn:M21}\\
			-{\bar{D}_2}\nabla v_\epsilon.\vec{n}&=0 \; \text{ on } S \times \partial\Omega, \label{eqn:M22}\\
			-{\bar{D}_2}\nabla v_\epsilon.\vec{n}&=\epsilon{\frac{\partial w_\epsilon}{\partial {t}}} \; \text{ on } S \times \Gamma_\epsilon^*, \label{eqn:M23}\\
			v_\epsilon(0,x)&=v_{I\epsilon}(x) \; \text{ in } \Omega_\epsilon^p, \label{eqn:M24}
		\end{align}
		\begin{align}
			\frac{\partial w_\epsilon}{\partial  {t}}&=k_d(R(u_\epsilon,v_\epsilon)-z_\epsilon) \; \text{ on } S \times \Gamma_\epsilon^*, \label{eqn:MN1} \\ 
			z_\epsilon&\in\psi(w_\epsilon) \; \text{ on } S \times \Gamma_\epsilon^*, \label{eqn:MN2}\\
			w_\epsilon(0,x)&=w_{I\epsilon}(x) \; \text{ on } \Gamma_\epsilon^*,\label{eqn:MN3}
		\end{align}
	\end{subequations}		 
	where $R : \mathbb{R}^2 \rightarrow [0, \infty)$ is defined by
	\begin{align}\label{eqn:RR}
		R(u_\epsilon, v_\epsilon)&\;\;=\;\;
		\left\{
		\begin{aligned}
			&k\frac{k_1u_\epsilon k_2 v_\epsilon}{(1+k_1u_\epsilon+k_2v_\epsilon)^2}&& \textnormal{ for } u_\epsilon>0, v_\epsilon>0,\\
			&0&&\textnormal{ otherwise }
		\end{aligned}\right.
	\end{align} 
and $k=\frac{k_f}{k_d}$. $k_1$ and $k_2$ are the Langmuir parameters for the mobile species $M_1$ and $M_2$. The rate of dissolution is given by $R_d=k_d\psi(w_\epsilon).$ Also, $k_f$ is the forward reaction rate constant for precipitation and $k_d$ denotes the dissolution rate constant. We denote this problem/model $(\ref{eqn:M11})-(\ref{eqn:MN3})$ by ($\mathcal{P}_\epsilon$).
   \section{Corrector estimate}
   Let $\theta \in [0,1]$ and $r,s\in\mathbb{R}$ be such that $\frac{1}{r}+\frac{1}{s}=1$. Suppose that $\Xi\in \{\Omega, \Omega^{p}_{\varepsilon}\}$, then $L^{r}(\Xi)$, $H^{1,r}(\Xi)$, $C^{\theta}(\overline{\Xi})$, $(\cdot,\cdot)_{\theta,r}$ and $[\cdot,\cdot]_{\theta}$ are the Lebesgue, Sobolev, H\"older, real- and complex-interpolation spaces respectively endowed with their standard norms. We define $\|f\|^r_{\Xi}=\int_{\Xi}|f|^r\,dx $ and $\|f\|^r_{(\Xi)^t}=\int_{S\times\Xi}|f|^r\,dx\,dt $. The symbols $\hookrightarrow$, $\hookrightarrow\hookrightarrow$ and $\underset{\hookrightarrow}{d}$ denote the continuous, compact and dense embeddings, respectively. We denote $L^{r}(\Xi)\hookrightarrow H^{1,s}(\Xi)^{*}$ as 
   \begin{eqnarray}
   	\langle f, v\rangle_{H^{1,s}(\Xi)^{*}\times H^{1,s}(\Xi)}&=&\langle f,v\rangle_{L^{r}(\Xi)\times L^{s}(\Xi)}:=\int_{\Xi}f v\,dx \textnormal{ for }f\in L^{r}(\Xi),\; v\in H^{1,s}(\Xi).\notag
   \end{eqnarray}
  We also introduce the $L^{r}(\Gamma_{\varepsilon}^*)-L^{s}(\Gamma_{\varepsilon}^*)$ duality as
  \begin{align*}
  	\langle \zeta_{1},\zeta_{2}\rangle:=\varepsilon\int_{\Gamma_{\varepsilon}^*}\zeta_{1}\zeta_{2}\,d\sigma_{x}\text{ for }\zeta_{1}\in L^{r}(\Gamma_{\varepsilon}^*), \; \zeta_{2}\in L^{s}(\Gamma_{\varepsilon}^*)
  \end{align*}
  and the space $L^{r}(S \times \Gamma_{\varepsilon}^*)$ is equipped with the norm
  \begin{align*}
  	\|\zeta\|^{r}_{(\Gamma_{\varepsilon}^*)^t}&\;\;:=\;\;
  	\left\{
  	\begin{aligned}
  		&\varepsilon\int_{S\times \Gamma_{\varepsilon}^*}|\zeta(t,x)|^{r}\,d\sigma_{x} \,dt&& \textnormal{for } 1\le r <\infty,\\
  		&\underset{(t,x) \in S\times \Gamma_{\varepsilon}^*}{\text{ess sup}}{|\zeta(t,x)|}&&\textnormal{for } r=\infty.
  	\end{aligned}\right.
  \end{align*} 
 For more details about the choice of the function spaces and the definition of the function spaces and embedding theorems for the problem $(\mathcal{P}_\epsilon)$ look into the work \cite{ghoshmahato}. We impose the following assumptions for the sake of analysis:\\
   $\bf A1.$ $u_{I\epsilon}(x), v_{I\epsilon}(x), w_{I\epsilon}(x) \geq 0.$ \quad $\bf A2.$ $R(u_\epsilon, v_\epsilon) = 0$  for all  $u_\epsilon\leq 0, v_\epsilon\leq0.$
   	$\bf A3.$ $R:\mathbb{R}^2\rightarrow[0,\infty)$ is Locally Lipschitz in $\mathbb{R}^2,$ as
   	\begin{align*}
   		|R(u_\epsilon^{(1)}, v_\epsilon)-R(u_\epsilon^{(2)}, v_\epsilon)|\leq L_R |u_\epsilon^{(1)}-u_\epsilon^{(2)}|,  	
   	\end{align*}
   	where $L_R>0$ is a constant. $L_R=\sup L_R(u_\epsilon, v_\epsilon)$ where $L_R(u_\epsilon,v_\epsilon)=kk_1k_2|v_\epsilon|(1+k_2v_\epsilon)^2+k_1^2u_\epsilon^{(1)}u_\epsilon^{(2)}|$.
   	$\bf A4.$ $u_{I\epsilon}, v_{I\epsilon}\in L^2(\Omega^p_\epsilon)$ and $w_{I\epsilon}\in L^\infty(\Gamma_\epsilon^*)$. 
   	$\bf A5.$ $\bar{D}_1=diag(D_1, D_1,..., D_1)$ and $\bar{D}_2=diag(D_2, D_2,..., D_2)$, where $D_1, D_2$ are positive constants.

   \subsection{Existence and Homogenization}
   \begin{thm}\label{thmex}
    Suppose the assumptions $(\bf A1.)-(\bf A5.)$ hold true, then there exists a unique positive weak solution $(u_\epsilon, v_\epsilon, w_\epsilon, z_\varepsilon)$ of $(\mathcal{P}_\varepsilon)$ which satisfies    
   \begin{subequations}
   	\begin{align}
   		& 0\leq \|u_{\epsilon}(t)\|_{\Omega_\epsilon^p}\leq M_u,  0\leq \|v_{\epsilon}(t)\|_{\Omega_\epsilon^p}\leq M_v\text{ a.e. in } S\times {\Omega_\epsilon^p}, \label{3.7a} \\
   		&0\leq \|w_{\epsilon}(t)\|_{\Gamma_\epsilon^*}\leq M_w, 0\leq z_\epsilon \leq 1 \text{ a.e. on } S\times {\Gamma_\epsilon^*}\text{ and }\label{3.7b}\\
   		&{\left\lVert u_{\epsilon}(t) \right\rVert}_{\Omega_\epsilon^p}+{D_1}{\left\lVert\nabla u_{\epsilon} \right\rVert}_{(\Omega_\epsilon^p)^t}+{\left\lVert \partial_t u_{\epsilon} \right\rVert}_{L^2(S ; H^{1,2}(\Omega_\epsilon^p)^*)} +{\left\lVert v_{\epsilon}(t) \right\rVert}_{\Omega_\epsilon^p}+{D_2}{\left\lVert\nabla v_{\epsilon} \right\rVert}_{(\Omega_\epsilon^p)^t}+\notag \\
   		&\qquad \qquad {\left\lVert \partial_t v_{\epsilon} \right\rVert}_{L^2(S ; H^{1,2}(\Omega_\epsilon^p)^*)} + {\left\lVert w_{\epsilon}(t) \right\rVert}_{\Gamma^*_\epsilon}+{\left\lVert \partial_t w_\epsilon\right\rVert}_{{(\Gamma^*_\epsilon)}^t}\leq C, \label{eqn:3.7c}
   	\end{align}   	
   \end{subequations}
   for a.e. $t\in S$, where $C$ is a generic constant independent of $\epsilon$. 
\end{thm}
\begin{proof}
	We employ Banach's fixed point theorem to establish the existence of the weak solution. The proof is done in \cite{ghoshmahato}.
\end{proof}
\begin{thm}\label{thm:UT}
Under the assumptions $({\bf A1.})- ({\bf A5.})$, there exist $(u_0,v_0,w_0,z_0)\in L^2(S;H^{1,2}(\Omega))\times L^2(S;H^{1,2}(\Omega))\times L^2(S;L^2(\Omega\times\Gamma))\times L^\infty(S\times\Omega\times\Gamma)$ such that $(u_0,v_0,w_0,z_0)$ is the unique solution of the problem
\begin{subequations}
	\begin{align}
		\frac{\partial u_0}{\partial t}-\nabla.(A\nabla u_0)+P(t,x)&=0 \quad\text{in}\quad S\times\Omega, \label{eqn:M1s1} \\
		-A\nabla u_0.\vec{n}&=0 \quad\text{on}\quad S\times \partial\Omega, \label{eqn:M1s2}\\
		u_0(0,x)=&u_{I0}(x)\quad\text{in}\quad \Omega, \label{eqn:M1s3}
	\end{align}
\end{subequations}\vspace{-0.8cm}
\begin{subequations}
	\begin{align}
		\frac{\partial v_0}{\partial t}-\nabla.(B\nabla v_0)+P(t,x)&=0 \quad\text{in}\quad S\times \Omega, \label{eqn:M2S1}\\
		-B\nabla v_0.\vec{n}&=0 \quad\text{on}\quad S\times  \partial\Omega,\label{eqn:M2S2}\\
		v_0(0,x)&=v_{I0}(x)\quad\text{in}\quad\Omega,\label{eqn:M2S3}
	\end{align}
\end{subequations}\vspace{-0.8cm}
\begin{subequations}
	\begin{align}
		&\frac{\partial w_0}{\partial t}=k_d(R(u_0,v_0)-z_0) \quad\text{in}\quad S\times \Omega \times \Gamma,  \label{eqn:M3S1}\\
		&\; \text{  where  } z_0\in \psi(w_0) \quad\text{in}\quad S\times \Omega \times \Gamma, \label{eqn:M3S2}\\
		&w_0(0,x,y)=w_{I0}(x,y) \quad\text{on}\quad \Omega \times \Gamma, \label{eqn:M3S3}
	\end{align}	
\end{subequations}
satisfying the a-priori bound
\begin{align}\label{eqn:MIE}
	\|u_0\|_{(\Omega)^t}&+\|\nabla u_0\|_{(\Omega)^t}+\|\partial_t u_0\|_{L^2(S; H^{1,2}(\Omega)^*)}+\|v_0\|_{(\Omega)^t}+\|\nabla v_0\|_{(\Omega)^t}+\|\partial_t v_0\|_{L^2(S; H^{1,2}(\Omega)^*)}\notag\\
	&+\|w_0\|_{(\Omega\times\Gamma)^t}+\|\partial_t w_0\|_{(\Omega\times\Gamma)^t}+\|P\|_{(\Omega)^t}\le C<\infty,
\end{align}
where 
\begin{align*}
	P(t,x)=\int_{\Gamma}\frac{1}{|Y^p|}\frac{\partial w_0}{\partial t}\;d\sigma_y
\end{align*}
and the elliptic homogenized matrix $A=(a_{ij})_{1\le i,j\le n}$ and $B=(b_{ij})_{1\le i,j\le n}$ are defined by
\begin{align*}
	a_{ij}=\int_{Y^p}\frac{D_1}{|Y^p|}\left(\delta_{ij}+\sum_{i,j=1}^{n}\frac{\partial l_j}{\partial y_i}\right)dy, \quad b_{ij}=\int_{Y^p}\frac{D_2}{|Y^p|}\left(\delta_{ij}+\sum_{i,j=1}^{n}\frac{\partial l_j}{\partial y_i}\right)dy.
\end{align*}
Moreover, $l_j\in H^{1,2}_{per}(Y^p)$ are the solutions of the cell problems 
\begin{align}\label{eqn:CP}
	\begin{cases}
		\nabla_y.(\nabla_y l_j+e_j)=0 \quad\text{for all}\quad y\in Y^p,\\
		(\nabla_y l_j+e_j).\vec{n}=0 \quad\text{on}\quad \Gamma,\\
		y \mapsto l_j(y) \text{ is } Y-\text{periodic},
	\end{cases}
\end{align}
for $j=1,2,\cdots,n$ and for almost every $x\in\Omega$.
\end{thm}
\begin{proof}
	We prove the theorem by using homogenization techniques such as two-scale convergence and boundary unfolding operator in \cite{ghoshmahato}.
\end{proof}
The main result of this paper is stated in the next Theorem.
\begin{thm}\label{thm:CT}
	Suppose that 
	\begin{align}
		&(a) u_{I\epsilon}\in H^{1,2}(\Omega_\epsilon^p),\label{eqn:C1}\\
		&(b) \lim_{\epsilon\rightarrow 0} \epsilon\int_{\Gamma_\epsilon^*} w^2_{I\epsilon}d\sigma_x=\int_{\Omega}\int_{\Gamma}w^2_{I0}dxd\sigma_y.\label{eqn:C2}
	\end{align}
	Now let $u_\epsilon, v_\epsilon, w_\epsilon$ be the solution of the micro problem $\eqref{eqn:M11}-\eqref{eqn:MN3}$ and $u_0, v_0, w_0$ are the solutions of the macro problem $\eqref{eqn:M1s1}-\eqref{eqn:M3S3}$ then the following convergence holds
	\begin{align}\label{eqn:MT}
		\begin{cases}
			(i)\|u_\epsilon-u_0\|_{C([0,T];L^2(\Omega_\epsilon^p))}\rightarrow 0, \quad (ii)\|\nabla u_\epsilon-C^\epsilon\nabla u_0\|_{L^2(0,T;L^2(\Omega_\epsilon^p))}\rightarrow 0,\\
			(iii)\|v_\epsilon-v_0\|_{C([0,T];L^2(\Omega_\epsilon^p))}\rightarrow 0, \quad (iv)\|\nabla v_\epsilon-C^\epsilon\nabla v_0\|_{L^2(0,T;L^2(\Omega_\epsilon^p))}\rightarrow 0.\\
			(v)\|w_\epsilon-w_0\|_{C([0,T];L^2(\Gamma_\epsilon^*))}\rightarrow 0.
		\end{cases}
	\end{align}
\end{thm}
  \subsection{Convergence of the energy}
   We define the energies associated with the mobile species $M_1$ of the micromodel and the macromodel as
   \begin{align}
   &E_\epsilon(t)=\frac{1}{2}\int_{\Omega_\epsilon^p}u^2_\epsilon(t)dx+\int_{0}^{t}\int_{\Omega_\epsilon^p}D_1\nabla u_\epsilon(\tau,x)\nabla u_\epsilon(\tau,x)dxd\tau,\label{eqn:E1}\\
   &E_0(t)=\frac{|Y^p|}{2}\int_{\Omega}u^2_0(t)dx+|Y^p|\int_{0}^{t}\int_{\Omega}A\nabla u_0(\tau,x)\nabla u_0(\tau,x)dxd\tau.\label{eqn:E2}
   \end{align}
   We choose $u_\epsilon(t,x), u_0(t,x)$ as the test functions in the equations \eqref{eqn:M11} and \eqref{eqn:M1s1} and see that the energy terms can be expressed as 
   \begin{align}
   	&E_\epsilon(t)=\frac{1}{2}\int_{\Omega_\epsilon^p}u_{I\epsilon}^2dx-\epsilon\int_{0}^{t}\int_{\Gamma_\epsilon^*}\frac{\partial w_\epsilon}{\partial t}(\tau,x)u_\epsilon(\tau,x)d\sigma_xd\tau,\label{eqn:E3}\\
   	&E_0(t)=\frac{|Y^p|}{2}\int_{\Omega}u_{I0}^2dx-\int_{0}^{t}\int_{\Omega}\int_{\Gamma}\frac{\partial w_0}{\partial t}(\tau,x,y)u_0(\tau,x)dxd\sigma_yd\tau.\label{eqn:E4}
   \end{align}
  The following convergence holds true for the energies:
   \begin{lemma}\label{lemma:EC}
   	$E_\epsilon(t)\rightarrow E_0(t)$ strongly $C([0,T])$ under the condition \eqref{eqn:C1}.
   \end{lemma}
   \begin{proof}
    We have to show that $E_\epsilon(t)\in C[0,T]$. As a consequence of Arzela-Ascoli theorem, it is equivalent to establish the followings:\\
    $(i) |E_\epsilon(t)|\le C$ for all $t\in [0,T]$.\\
    $(ii) |E_\epsilon(t+h)-E_\epsilon(t)|\le \theta(h)$ uniformly with respect to $\epsilon$ for all $t\in[0,T-h)$ for all $h>0$ and $\theta(h)\rightarrow 0$ as $h\rightarrow 0$. \\
    We use Theorem \ref{thmex} to estimate $E_\epsilon(t)$ and obtain
    \begin{align*}
    	|E_\epsilon(t)|\le \frac{1}{2}\|u_{I\epsilon}\|^2_{\Omega_\epsilon^p}+\frac{k_d^2}{2}(1+\frac{k}{4})^2\frac{T|\Gamma||\Omega|}{|Y|}+\frac{C}{2}\left[\|u_\epsilon\|^2_{(\Omega_\epsilon^p)^t}+\|\nabla u_\epsilon\|^2_{(\Omega_\epsilon^p)^t}\right]=C.
    \end{align*}
    Hence $(i)$ proved. Next, to show $(ii)$ we consider \eqref{eqn:E3} and utilize the a-priori bounds of Theorem \ref{thmex} and deduce
    \begin{align*}
    	\left|E_\epsilon(t+h)-E_\epsilon(t)\right|&=\left|\epsilon\int_{t}^{t+h}\int_{\Gamma_\epsilon^*}\frac{\partial w_\epsilon}{\partial t}u_\epsilon d\sigma_xd\tau\right|\\
    	&\le C h^{\frac{1}{2}}\|u_\epsilon\|_{L^\infty(S;H^{1,2}(\Omega_\epsilon^p))}\left\lVert \frac{\partial w_\epsilon}{\partial t}\right\rVert_{L^2(S\times\Gamma_\epsilon^*)}=C_1h^{\frac{1}{2}}.
    \end{align*}
    Therefore upto a subsequence $E_\epsilon(t)\rightarrow \xi$ strongly in $C[0,T]$. It remains to prove that $\xi=E_0(t)$. To do so we pass the limit in $E_\epsilon(t)$ and by Lemma $6.2$ of \cite{ghoshmahato}, we get 
    \begin{align*}
    	\lim_{\epsilon\rightarrow 0}E_\epsilon(t)&=\frac{1}{2}\lim_{\epsilon\rightarrow 0}\int_{\Omega_\epsilon^p}u_{I\epsilon}^2dx-\lim_{\epsilon\rightarrow 0}\epsilon\int_{0}^{t}\int_{\Gamma_\epsilon^*}\frac{\partial w_\epsilon}{\partial t}(\tau,x)u_\epsilon(\tau,x)d\sigma_xd\tau\\
    	&=\frac{|Y^p|}{2}\int_{\Omega}u_{I0}^2dx-\int_{0}^{t}\int_{\Omega}\int_{\Gamma}\frac{\partial w_0}{\partial t}(\tau,x,y)u_0(\tau,x)dxd\sigma_yd\tau \\
    	&=E_0(t).
    \end{align*}
   \end{proof}
   \subsection{Derivation of the corrector estimate}
   The corrector matrix $C^\epsilon=(C^\epsilon_{ij})_{1\le i,j \le n}$ is defined by
   \begin{align*}
   	\begin{cases}
   		C^\epsilon_{ij}(x)=c_{ij}(\frac{x}{\epsilon}) \text{ a.e. on }\Omega_\epsilon^p,\\
   		C_{ij}(y)=\delta_{ij}+\frac{\partial l_j(y)}{\partial y_i}=\frac{\partial h^\epsilon_j}{\partial y_i}(y)\text{ a.e. on } Y^p,
   	\end{cases}
   \end{align*}
 where $l_j$ is given by \eqref{eqn:CP} and $h^\epsilon_j=e_j.y+l_j(y)\in H^{1,2}_{per}(Y^p)$ is $Y-$periodic and satisfies
 \begin{align}
 	\begin{cases}
 		\nabla_y.(\nabla_y h^\epsilon_j)=0 \quad\text{for all}\quad y\in Y^p,\\
 		\nabla_y h^\epsilon_j.\vec{n}=0 \quad\text{on}\quad \Gamma,\\
 		(h^\epsilon_j-e_j.y) \text{ is } Y-\text{periodic},
 	\end{cases}
 \end{align}
 for $j=1,2,\cdots,n$.
 \begin{lemma}
 	The sequence $\{h^\epsilon_j\}$ is weakly convergent to $h_j$ in $ H^{1,2}_{per}(Y^p)$, where $h_j$ is the solution of 
 	\begin{align}
 		\begin{cases}
 			\nabla_y.(\nabla_y h_j)=0 \quad\text{for all}\quad y\in Y^p,\\
 			\nabla_y h_j.\vec{n}=0 \quad\text{on}\quad \Gamma,\\
 			(h_j-e_j.y) \text{ is } Y-\text{periodic},
 		\end{cases}
 	\end{align}
 for $j=1,2,\cdots,n$.
 \end{lemma}
 \begin{proof}
 	We can see that $\|h^\epsilon_j\|_{H^{1,2}_{per}(Y^p)}\le C$, where $C$ is a constant independent of $\epsilon$. Therefore, we can extract a subsequence(denoted by the same notation) such that $h^\epsilon_j\rightharpoonup h_j$ weakly in $ H^{1,2}_{per}(Y^p)$.
 \end{proof}
 Next, we set
 \begin{align*}
 	a^\epsilon_j(x)=e_j.x+\epsilon (Q_j(l_j))(\frac{x}{\epsilon}),
 \end{align*}
 where the extension operators $Q_j$ are defined in Lemma $3.1$ of \cite{donato2004homogenization} and $l_j$ are the solutions of the cell problems. Then we have by standard arguments (see for instance, \cite{cioranescu1979homogenization})
 \begin{align}\label{eqn:CR}
 	\begin{cases}
 		a^\epsilon_j \rightharpoonup e_j.x \text{ weakly in } H^{1,2}(\Omega),\\
 		a^\epsilon_j \rightarrow e_j.x \text{ strongly in } L^2(\Omega),
 	\end{cases}
 \end{align}
 due to the periodicity of these functions. We can derive that $C^\epsilon\rightharpoonup I$ weakly in $(L^2(\Omega))^{n\times n}$. Now, we denote
 \begin{align}\label{eqn:E}
 	\eta_i^\epsilon=\left(D_1\frac{\partial a_j^\epsilon}{\partial x_1},D_1\frac{\partial a_j^\epsilon}{\partial x_2},\cdots, D_1\frac{\partial a_j^\epsilon}{\partial x_n}\right)=D_1\nabla a^\epsilon_j.
 \end{align}
 \begin{lemma}\label{lemma:MC}
 	Let $\eta_i^\epsilon$ be as \eqref{eqn:E} and   $\tilde{\eta_j^\epsilon}$ denotes the zero extension to the whole domain $\Omega$. Then         $\tilde{\eta_j^\epsilon} \rightharpoonup \int_{Y^p} D_1\nabla h_jdy=A|Y^p|e_j$ weakly in $(L^2(\Omega))^n$.
 \end{lemma}
 \begin{proof}
 	Since $\eta_j^\epsilon$ is bounded in $L^2(\Omega_\epsilon^p)$ and $\tilde{\eta_j^\epsilon}=D_1\tilde{h_j^\epsilon}$ and $D_1\tilde{h_j^\epsilon}$ is $Y-$ periodic so $\tilde{\eta_j^\epsilon} \rightharpoonup \mathcal{M}_Y(D_1\tilde{h_j^\epsilon})$ weakly in $(L^2(\Omega))^n$. We now apply Lemma $3.4$ of \cite{donato2004homogenization} and conclude that
 	\begin{align*}
 		\mathcal{M}_Y(D_1\tilde{h_j^\epsilon})\rightharpoonup \mathcal{M}_Y(D_1\tilde{h_j}) \text{ weakly in } (L^2(\Omega))^n.
 	\end{align*}
 \end{proof}
 Moreover, it can be seen that $\eta_j^\epsilon$ satisfies the system
\begin{align}\label{eqn:2}
	\begin{cases}
		\nabla.\eta_j^\epsilon=0\text{ in } \Omega_\epsilon^p,\\
		\eta_j^\epsilon.\vec{n}=0 \text{ on } \Gamma_\epsilon^*.
	\end{cases}
\end{align}
  \begin{lemma}\label{lemma:RC}
  	Under the assumptions of Theorem \ref{thm:CT}, for any $\Phi\in C^\infty([0,T];\mathcal{D}(\Omega))$, set
  	\begin{align*}
  		\rho_\epsilon(t)=\frac{1}{2}\|u_\epsilon(t)-\Phi(t)\|^2_{\Omega_\epsilon^p}+\int_{0}^{t}\int_{\Omega_\epsilon^p}D_1(\nabla u_\epsilon-C^\epsilon\nabla \Phi)(\tau,x)(\nabla u_\epsilon-C^\epsilon\nabla \Phi)(\tau,x)dxd\tau.
  	\end{align*}
  Then $\rho_\epsilon(t)\rightarrow \rho_0(t)$ strongly $C[0,T]$, where
  \begin{align*}
  	\rho_0(t)=\frac{|Y^p|}{2}\|u_0(t)-\Phi(t)\|^2_{\Omega}+\int_{0}^{t}\int_{\Omega}|Y^p|A(\nabla u_0-\nabla \Phi)(\tau,x)(\nabla u_0-\nabla \Phi)(\tau,x)dxd\tau.
  \end{align*}
\end{lemma}
   \begin{proof}
   	We can write $\rho_\epsilon(t)$ as 
   	\begin{align*}
   		\rho_\epsilon(t)=\rho_\epsilon^1(t)+\rho_\epsilon^2(t)-\rho_\epsilon^3(t),
   	\end{align*}
   where 
   \begin{align*}
   		&\rho_\epsilon^1(t)=\frac{1}{2}\|u_\epsilon(t)\|^2_{\Omega_\epsilon^p}+\int_{0}^{t}\int_{\Omega_\epsilon^p}D_1\nabla u_\epsilon(\tau,x)\nabla u_\epsilon(\tau,x) dxd\tau,\\
     	&\rho_\epsilon^2(t)=\frac{1}{2}\|\Phi(t)\|^2_{\Omega_\epsilon^p}+\int_{0}^{t}\int_{\Omega_\epsilon^p}D_1(C^\epsilon \nabla\Phi)(C^\epsilon \nabla\Phi) dxd\tau,\\
     	&\rho_\epsilon^3(t)=\int_{\Omega_\epsilon^p}u_\epsilon(t)\Phi(t)dx+\int_{0}^{t}\int_{\Omega_\epsilon^p}D_1\nabla u_\epsilon(C^\epsilon\nabla \Phi)dxd\tau+\int_{0}^{t}\int_{\Omega_\epsilon^p}D_1(C^\epsilon\nabla\Phi)\nabla u_\epsilon dx d\tau.
   \end{align*}
  We now pass the limits to each term separately. Lemma \ref{lemma:EC} implies
  \begin{align}\label{eqn:R1}
  	\rho_\epsilon^1(t)=E_\epsilon(t)\rightarrow E_0(t)=\frac{|Y^p|}{2}\int_{\Omega}u_0^2(t)dx+|Y^p|\int_{0}^{t}\int_{\Omega}A\nabla u_0(\tau,x)\nabla u_0(\tau,x)dxd\tau \text{ in } C[0,T].
  \end{align}
 We first establish the point-wise convergence of $\rho_\epsilon^2$. So basically we need to show the point-wise convergent of the second term of $\rho_\epsilon^2$. That means, we have to calculate
 \begin{align*}
 	&\lim_{\epsilon\rightarrow 0}\int_{0}^{t}\int_{\Omega_\epsilon^p}D_1(C^\epsilon \nabla\Phi)(\tau,x)(C^\epsilon \nabla\Phi)(\tau,x) dxd\tau=\sum_{i,j=1}^{n}\lim_{\epsilon\rightarrow 0}\int_{0}^{t}\int_{\Omega}\chi(\frac{x}{\epsilon})\eta_i^\epsilon\nabla a_j^\epsilon\Phi_{x_i}\Phi_{x_j}dxd\tau\\
 	&=\sum_{i,j=1}^{n}\lim_{\epsilon\rightarrow 0}\int_{0}^{t}\int_{\Omega}\chi(\frac{x}{\epsilon})\eta_i^\epsilon\nabla (a_j^\epsilon\Phi_{x_i}\Phi_{x_j})dxd\tau-\sum_{i,j=1}^{n}\lim_{\epsilon\rightarrow 0}\int_{0}^{t}\int_{\Omega}\chi(\frac{x}{\epsilon})\eta_i^\epsilon a_j^\epsilon\nabla(\Phi_{x_i}\Phi_{x_j})dxd\tau.
 \end{align*}
 Using \eqref{eqn:CR}, \eqref{eqn:2} and Lemma \ref{lemma:MC}, we obtain
 \begin{align}\label{eqn:1}
 	\rho_\epsilon^2\rightarrow \frac{|Y^p|}{2}\|\Phi(t)\|^2_{\Omega}+|Y^p|\int_{0}^{t}\int_{\Omega}A\nabla\Phi\nabla\Phi dxd\tau, \text{ for any } t\in [0,T].
 \end{align} 
 Next we need to show $\rho_\epsilon^2$ belongs to a compact set in $C[0,T]$. Due to the compact injection $H^{1,\infty}(0,T)\hookrightarrow\hookrightarrow C([0,T])$, it is equivalent to prove
 \begin{align*}
 	\|\rho_\epsilon^2\|_{L^\infty(0,T)}+\left\lVert\partial_t (\rho_\epsilon^2)\right\rVert_{L^\infty(0,T)}\le C,
 \end{align*}
 where $C$ is a constant independent of $\epsilon$. This follows immediately due to the weak convergence of $C^\epsilon$ and the fact that $\Phi$ is regular and independent of $\epsilon$. This in combination with \eqref{eqn:1} implies that
 \begin{align}\label{eqn:3}
 		\rho_\epsilon^2\rightarrow \frac{|Y^p|}{2}\|\Phi(t)\|^2_{\Omega}+|Y^p|\int_{0}^{t}\int_{\Omega}A\nabla\Phi\nabla\Phi dxd\tau \text{ in } C[0,T].
 \end{align}
 We proceed similarly for $\rho_\epsilon^3$. Mainly, we have to show the point-wise convergence of the second and third term of $\rho_\epsilon^3$. We start with the second term and get
 \begin{align*}
 	&\lim_{\epsilon\rightarrow 0} \int_{0}^{t}\int_{\Omega_\epsilon^p}D_1\nabla u_\epsilon(C^\epsilon\nabla \Phi)dxd\tau=\sum_{i=1}^{n}\lim_{\epsilon\rightarrow 0}\int_{0}^{t}\int_{\Omega}\chi(\frac{x}{\epsilon})D_1\nabla u_\epsilon\nabla a_i^\epsilon\Phi_{x_i}dxd\tau \\
 	&=\sum_{i=1}^{n}\lim_{\epsilon\rightarrow 0}\int_{0}^{t}\int_{\Omega}\chi(\frac{x}{\epsilon})D_1\nabla u_\epsilon\nabla(a_i^\epsilon\Phi_{x_i})dxd\tau-\lim_{\epsilon\rightarrow 0}\sum_{i=1}^{n}\int_{0}^{t}\int_{\Omega}\chi(\frac{x}{\epsilon})D_1\nabla u_\epsilon a_i^\epsilon\nabla\Phi_{x_i}dxd\tau.
 \end{align*}
 We now choose $a_i^\epsilon\Phi_{x_i}$ as a test function in the variational formulation of  \eqref{eqn:M11} and pass the homogenization limit to zero. Consequently, using the convergence results of Lemma 6.2 of \cite{ghoshmahato} and \eqref{eqn:CR}, we obtain
 \begin{align*}
 	\int_{0}^{t}\int_{\Omega_\epsilon^p}D_1\nabla u_\epsilon(C^\epsilon\nabla \Phi)dxd\tau\rightarrow |Y^p|\int_{0}^{t}\int_{\Omega}A\nabla u_0\nabla\Phi dxd\tau.
 \end{align*}
  We do the same calculation for the third term of $\rho_\epsilon^3$ just like we did for the second term of $\rho_\epsilon^2$ and combining all the terms we are led to
 \begin{align*}
 	\rho_\epsilon^3\rightarrow |Y^p|\int_{\Omega}u_0\Phi dx+|Y^p|\int_{0}^{t}\int_{\Omega}A\nabla u_0\nabla\Phi dxd\tau+|Y^p|\int_{0}^{t}\int_{\Omega}A\nabla\Phi\nabla u_0dxd\tau,
 \end{align*} 
 for any $t\in[0,T]$. We now prove that $\rho_\epsilon^3$ is bounded in $H^{1,2}(0,T)$, that is
 \begin{align*}
 	\|\rho_\epsilon^3\|_{L^\infty(0,T)}+\left\lVert\partial_t (\rho_\epsilon^3)\right\rVert_{L^2(0,T)}\le C.
 \end{align*}
 This comes from the the a-priori estimates of Theorem \ref{thmex}, weak convergence of $C^\epsilon$ and the regularity of $\Phi$. Further, due to the compact injection $H^{1,2}(0,T)\hookrightarrow\hookrightarrow C[0,T]$, we can write
 \begin{align}\label{eqn:R3}
 	\rho_\epsilon^3\rightarrow |Y^p|\int_{\Omega}u_0\Phi dx+|Y^p|\int_{0}^{t}\int_{\Omega}A\nabla u_0\nabla\Phi dxd\tau+|Y^p|\int_{0}^{t}\int_{\Omega}A\nabla\Phi\nabla u_0dxd\tau\text{ in } C[0,T].
 \end{align}
 Finally recalling that $\rho_\epsilon(t)=\rho_\epsilon^1(t)+\rho_\epsilon^2(t)-\rho_\epsilon^3(t)$ and  the convergences \eqref{eqn:R1}, \eqref{eqn:3} and \eqref{eqn:R3} gives the desired result.
\end{proof}
  \begin{proof}[\textit{\textbf{Proof of Theorem \ref{thm:CT}}}]
  	As $u_0\in L^2(S;H^{1,2}(\Omega))\cap C([0,T];L^2(\Omega))$, so by density arguments(see \cite{cioranescu1999introduction}), for any $\zeta>0$, there exists $\Phi_\zeta\in C^\infty([0,T];D(\Omega))$ such that
  	\begin{align}\label{eqn:BE}
  		\begin{cases}
  			&(i)\|u_0-\Phi_\zeta\|^2_{C([0,T];L^2(\Omega))}\le\zeta, \\
  			&(ii)\|\nabla u_0-\nabla\Phi_\zeta\|^2_{L^2(S;L^2(\Omega))}\le\zeta.
  		\end{cases}
  	\end{align}
   We now use \eqref{eqn:BE} and get
   \begin{align}
   	\|u_\epsilon-u_0\|^2_{C([0,T];L^2(\Omega^p_\epsilon))}&\le 2\|u_\epsilon-\Phi_\zeta\|^2_{C([0,T];L^2(\Omega^p_\epsilon))}+2\|u_0-\Phi_\zeta\|^2_{C([0,T];L^2(\Omega^p_\epsilon))}\nonumber\\
   	&\le 2\|u_\epsilon-\Phi_\zeta\|^2_{C([0,T];L^2(\Omega_\epsilon^p))}+2\zeta.\label{eqn:UE}
   \end{align}
  Then we need to evaluate $\|u_\epsilon-\Phi_\zeta\|^2_{C([0,T];L^2(\Omega_\epsilon^p))}$. For that, we set
  \begin{align}\label{eqn:R}
  	\rho_\epsilon^\zeta(t)=\frac{1}{2}\|u_\epsilon(t)-\Phi_\zeta(t)\|^2_{\Omega_\epsilon^p}+\int_{0}^{t}\int_{\Omega_\epsilon^p}D_1(\nabla u_\epsilon-C^\epsilon\nabla\Phi_\zeta)(\tau,x)(\nabla u_\epsilon-C^\epsilon\nabla\Phi_\zeta)(\tau,x)dxd\tau.
  \end{align}
  By Lemma \ref{lemma:RC}, we obtain
  \begin{align*}
  	\|\rho^\zeta\|_{C[0,T]}=\limsup_{\epsilon\rightarrow0}\|\rho_\epsilon^\zeta\|_{C[0,T]}\ge\frac{1}{2}\limsup_{\epsilon\rightarrow0}\|u_\epsilon(t)-\Phi_\zeta(t)\|^2_{\Omega_\epsilon^p},
  \end{align*}
  where
  \begin{align*}
  	\rho^\zeta&=\frac{|Y^p|}{2}\|u_0(t)-\Phi_\zeta(t)\|_{\Omega}+\int_{0}^{t}\int_{\Omega}|Y^p|A(\nabla u_0-\nabla \Phi)(\tau,x)(\nabla u_0-\nabla \Phi)(\tau,x)dxd\tau.
  \end{align*}
  Next we apply \eqref{eqn:BE} and the fact that $A$ is bounded to derive the estimate
  \begin{align}\label{eqn:RB}
  	\|\rho^\zeta\|_{C[0,T]}\le 
  	\frac{\zeta|Y^p|}{2}+M|Y^p|\zeta=\zeta|Y^p|(M+\frac{1}{2}).
  \end{align}
 Therefore, \eqref{eqn:UE} takes the form
  \begin{align*}
  	\limsup_{\epsilon\rightarrow0}\|u_\epsilon-u_0\|^2_{C([0,T];L^2(\Omega_\epsilon^p))}\le 4\|\rho^\zeta\|_{C[0,T]}+2\zeta\le 2\zeta(1+|Y^p|(2M+1)).
  \end{align*}
  This implies $(i)$ of \eqref{eqn:MT} since $\zeta$ is arbitrary. We now write
  \begin{align*}
  \nabla u_\epsilon-C^\epsilon\nabla u_0=(\nabla u_\epsilon-C^\epsilon\nabla\Phi_\zeta)+C^\epsilon(\nabla \Phi_\zeta-\nabla u_0).
  \end{align*}
 Then due to the weak convergence of $C^\epsilon$ and \eqref{eqn:BE}, we have
 \begin{align}
 	\lim_{\epsilon\rightarrow 0}&\int_{0}^{T}\|\nabla u_\epsilon(t)-C^\epsilon\nabla u_0(t)\|^2_{\Omega_\epsilon^p}dt\nonumber\\
 	&\le2\limsup_{\epsilon\rightarrow0}\int_{0}^{T}\|\nabla u_\epsilon(t)-C^\epsilon\nabla\Phi_\zeta(t)\|^2_{\Omega_\epsilon^p}dt+2\limsup_{\epsilon\rightarrow0}\int_{0}^{T}\|C^\epsilon\|_{\Omega^p_\epsilon}\|\nabla\Phi_\zeta-\nabla u_0\|^2_{\Omega_\epsilon^p}dt\nonumber\\
 	&\le 2\limsup_{\epsilon\rightarrow0}\int_{0}^{T}\|\nabla u_\epsilon(t)-C^\epsilon\nabla\Phi_\zeta(t)\|^2_{\Omega_\epsilon^p}dt+2C_1\zeta.\label{eqn:GUB}
 \end{align}
 To find out the estimate for the integral form in the right hand side, we first rewritten \eqref{eqn:R} for $t=T$. Then application of the Lemma \ref{lemma:RC} and the inequality \eqref{eqn:RB} leads to
 \begin{align*}
 	\limsup_{\epsilon\rightarrow0}\int_{0}^{T}\|\nabla u_\epsilon(t)-C^\epsilon\nabla\Phi_\zeta(t)\|^2_{\Omega_\epsilon^p}dt\le\frac{1}{D_1}\lim_{\epsilon\rightarrow 0}\rho_\epsilon^\zeta(T)=\frac{1}{D_1}\rho^\zeta(T)=\frac{\zeta|Y^p|}{2D_1}(2M+1).
 \end{align*}
 After that, we substitute this in \eqref{eqn:GUB} and obtain $(ii)$ of \eqref{eqn:MT}. We can establish $(iii)$ and $(iv)$ of \eqref{eqn:MT} by taking the corresponding microscopic and macroscopic equations for the mobile species $M_2$ and following the same line of arguments. For the case of immobile species, we first subtract \eqref{eqn:MN1} from \eqref{eqn:M3S1}. Then multiplication by $(w_\epsilon(t)-w_0(t))$ and integration over $S\times\Gamma_\epsilon^*$ yields
 \begin{align*}
 	\int_{0}^{t}\frac{\partial}{\partial t}\|w_\epsilon(s)-w_0(s)\|^2_{\Gamma_\epsilon^*}ds=2k_d\epsilon\int_{0}^{t}\int_{\Gamma_\epsilon^*}(R(u_\epsilon,v_\epsilon)-R(u_0,v_0)-z_\epsilon+z_0)(w_\epsilon-w_0)d\sigma_xdt.
 \end{align*}
 Since $z_\epsilon,z_0$ is monotone with respect to $w_\epsilon, w_0$, so we can write
 \begin{align*}
 	(z_\epsilon-z_0)(w_\epsilon-w_0)\ge 0.
 \end{align*}
 Therefore, we obtain the inequality
 \begin{align*}
 	\|w_\epsilon(t)-w_0(t)\|^2_{\Gamma_\epsilon^*}\le \|w_{I\epsilon}-w_{I0}\|^2_{\Gamma_\epsilon^*}+k_d^2\|R(u_\epsilon,v_\epsilon)-R(u_0,v_0))\|^2_{\Gamma_\epsilon^*}+\int_{0}^{t}\|w_\epsilon(s)-w_0(s)\|^2_{\Gamma_\epsilon^*}ds.
 \end{align*}
 We now use Lemma $6.3$ of \cite{ghoshmahato} and Gronwall's inequality to get
 \begin{align*}
 	\|w_\epsilon(t)-w_0(t)\|^2_{\Gamma_\epsilon^*}\le \|w_{I\epsilon}-w_{I0}\|^2_{\Gamma_\epsilon^*} e^T.
 \end{align*}
 Consequently, \eqref{eqn:C2} gives the desired convergence $(v)$. This concludes the proof.
 \end{proof}

  \section{Numerical simulation}
  \textbf{The physics setting:} In this section we compare the numerical solutions of the microscopic equations $\eqref{eqn:M11}-\eqref{eqn:MN3}$ with the numerical solutions of the macroscopic equations $\eqref{eqn:M1s1}-\eqref{eqn:M3S3}$ in order to see how well the homogenized equations approximate the averaged behavior of the original model. To achieve this goal, we start with the domain as $\Omega:=[0,1.2]\times[0,1]$ in $\mathbb{R}^2$. The unit representative cell is denoted by $Y=[0,1]\times[0,1]\subset \mathbb{R}^2$ which consist the solid part $Y^s=B((0.5,0.5),0.25)$. We choose the scaling parameter $\epsilon=0.2$. We perform numerical experiments by using COMSOL\cite{comsol}. Further, two mobile species $M_1$ and $M_2$ are present in $\Omega_\epsilon^p$ and one immobile species $M_{12}$ is present on the interface $\Gamma_\epsilon^*$. They are connected via the reversible reaction 
  \begin{equation}\label{eqn:RE}
  	M_1+M_2\leftrightarrow M_{12} \qquad \text{ on }\quad \Gamma_\varepsilon^*,
  \end{equation}
   where the reaction rate term is given by \eqref{eqn:RR}.
   \subsection{Simulation of the micromodel}
   Let the molar concentrations of $M_1, M_2$ and $M_{12}$ are given by  $u_\epsilon, v_\epsilon$ and $w_\epsilon$, respectively. We choose the parameter values and the regularization parameter as
   \begin{align}\label{eqn:PV}
   	D_1=1, D_2=2, k_f=1.8, k_d=2.2, k_1=1, k_2=1, \delta=0.01
   \end{align}
   and for the initial conditions we use $u_\epsilon(0,x,y)=5(x+y), v_\epsilon(0,x)=8x+2y$ and $w_\epsilon(0,x)=3x+y$. We choose ``Normal'' mesh available in COMSOL to discretize the domain $\Omega_\epsilon^p$. We solve the system for $t=20s$.
   \begin{figure}[H]
   	\centering
   	\subfloat[\centering Concentration of $M_1$ for $t=5s$]{{\includegraphics[width=7cm]{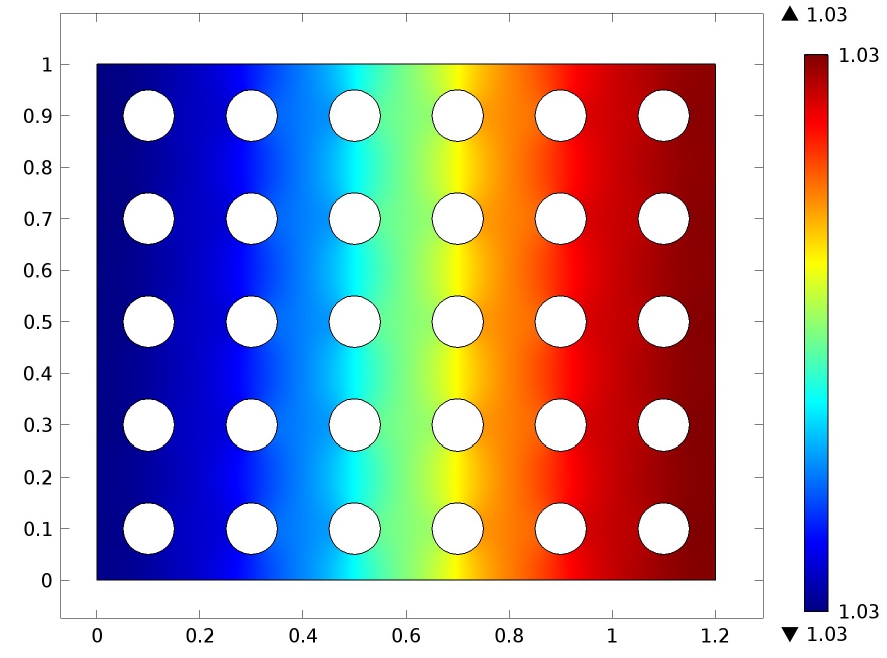} }}
   	\qquad
   	\subfloat[\centering  Concentration of $M_1$ for $t=10s$]{{\includegraphics[width=7cm]{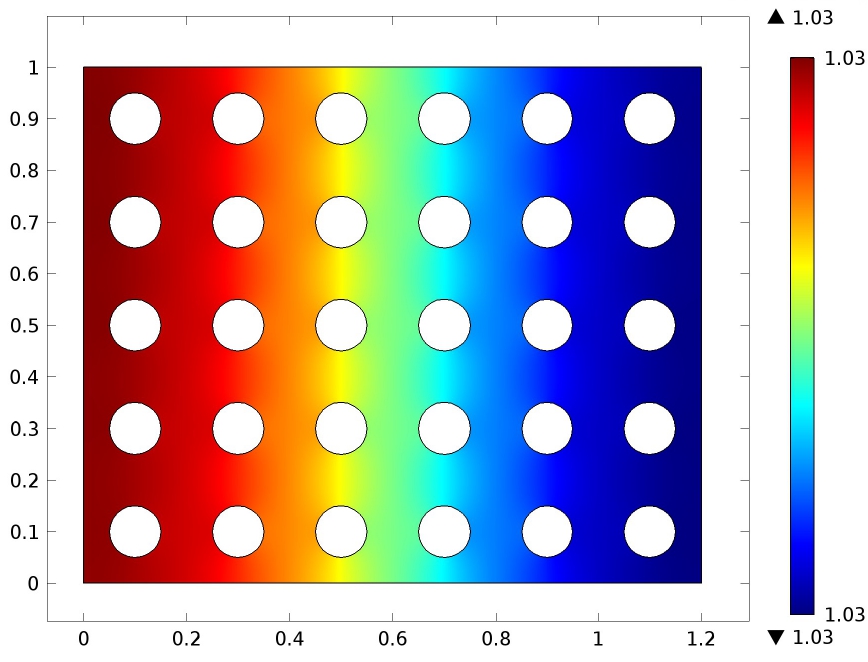} }}\\
   	\centering
   	\subfloat[\centering Concentration of $M_1$ for $t=15s$]{{\includegraphics[width=7cm]{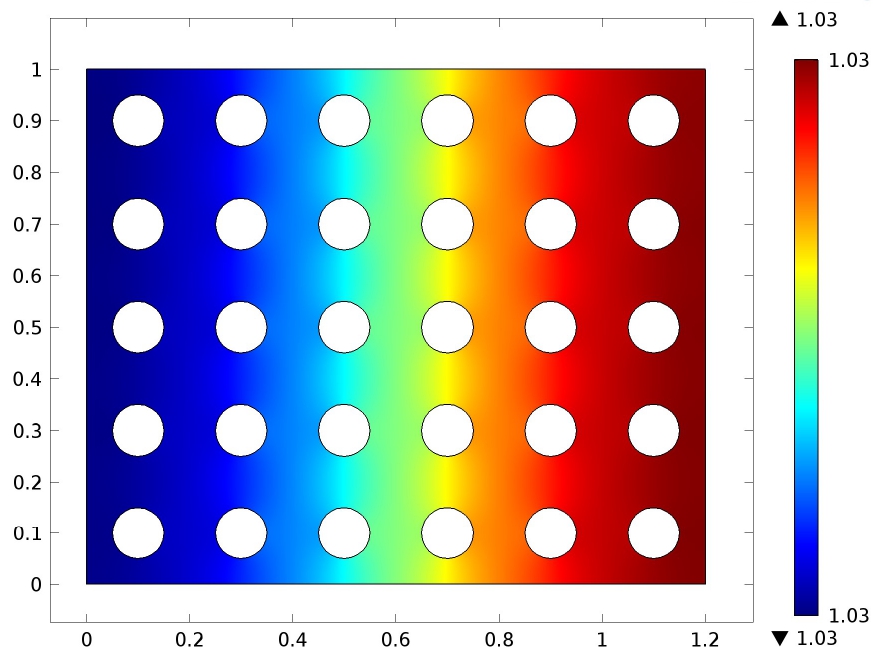} }}
   	\qquad
   	\subfloat[\centering  Concentration of $M_1$ for $t=20s$]{{\includegraphics[width=7cm]{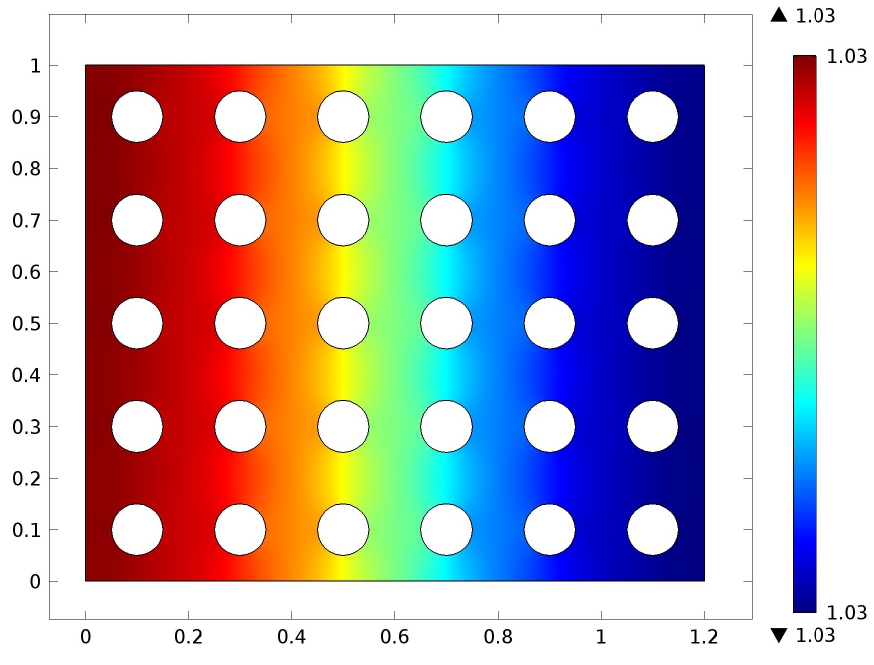} }}
   	\caption{Concentration of the first mobile species $M_1$ in $\Omega_\epsilon^p$ for different time.}
   \end{figure}
   \begin{figure}[h!]
   	\begin{center}
   		\includegraphics[width=10cm, height=6.5cm]{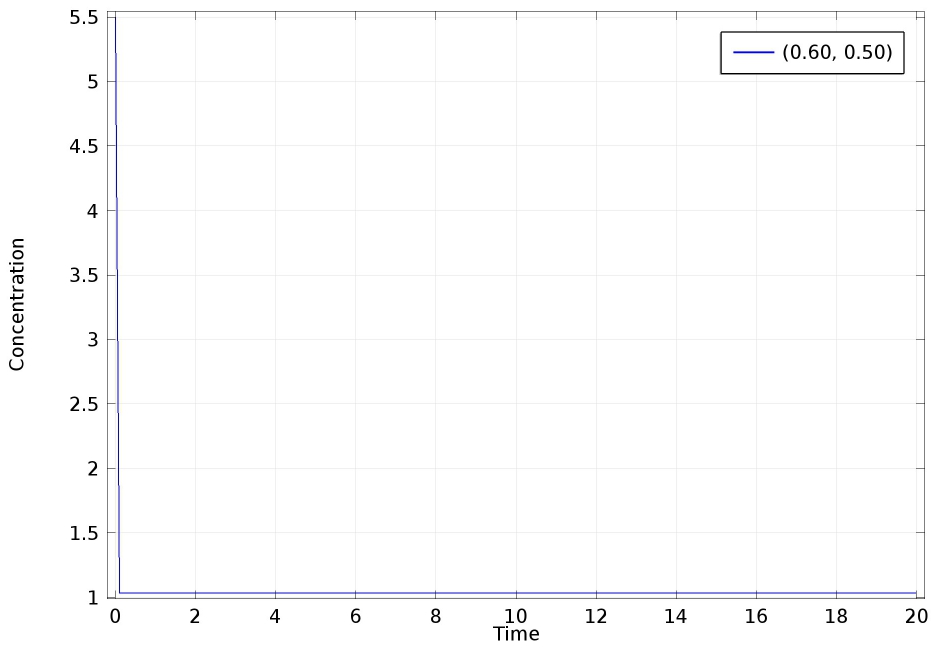}
   		\caption{Concentration of $M_1$ at the point $(0.6,0.5)$ in $\Omega_\epsilon^p$ in $20s$.}
   	\end{center}
   \end{figure}
   We notice that the time taken by the solver is $10s$. The concentration of $M_1$ for $t=5s, t=10s, t=15s$ and $t=20s$ is depicted in Figure $2$. We also plotted the change of concentration of $M_1$ at $(0.6,0.5)$ for $20s$ in Figure $3.$ We can see that there is a jump in concentration at $t=0s$ and it attains the value $5.5$. Whereas at $t=0.1s$ it became $1.03$ so the reaction tries to stabilize it. 
  \begin{figure}[H]
  	\centering
  	\subfloat[\centering Concentration of $M_2$ for $t=5s$]{{\includegraphics[width=7cm]{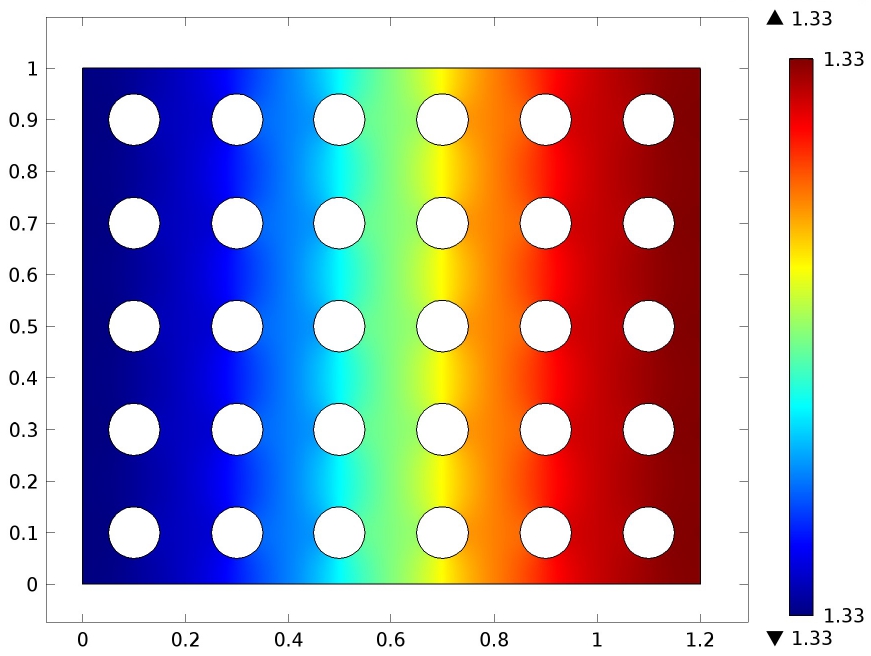} }}
  	\qquad
  	\subfloat[\centering  Concentration of $M_2$ for $t=10s$]{{\includegraphics[width=7cm]{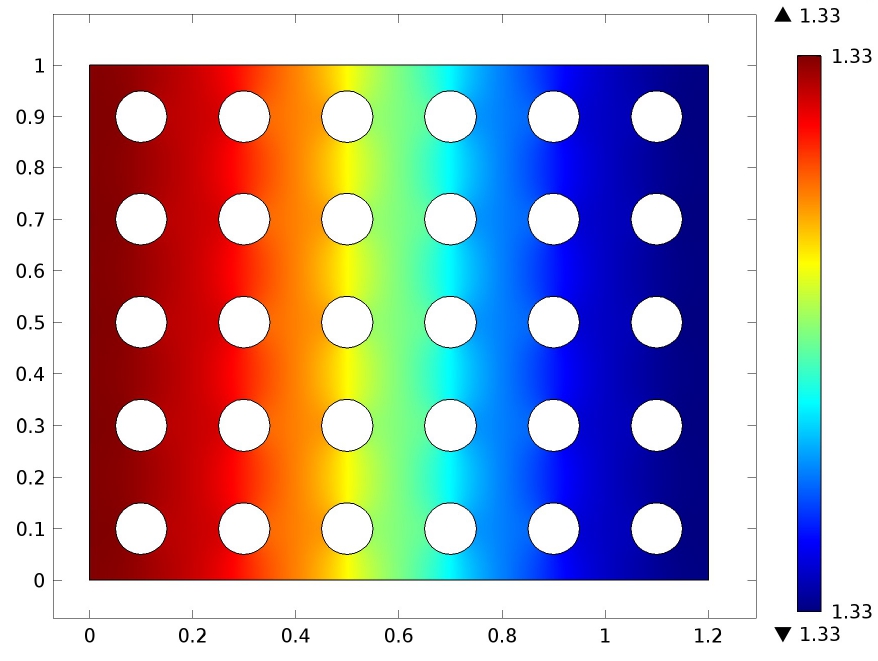} }}\\
  	\centering
  	\subfloat[\centering Concentration of $M_2$ for $t=15s$]{{\includegraphics[width=7cm]{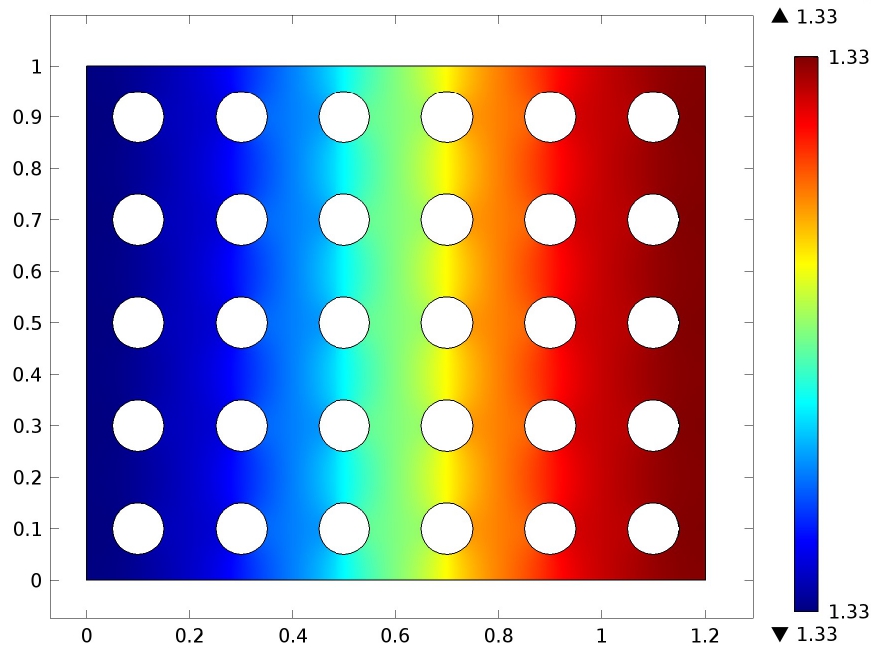} }}
  	\qquad
  	\subfloat[\centering  Concentration of $M_2$ for $t=20s$]{{\includegraphics[width=7cm]{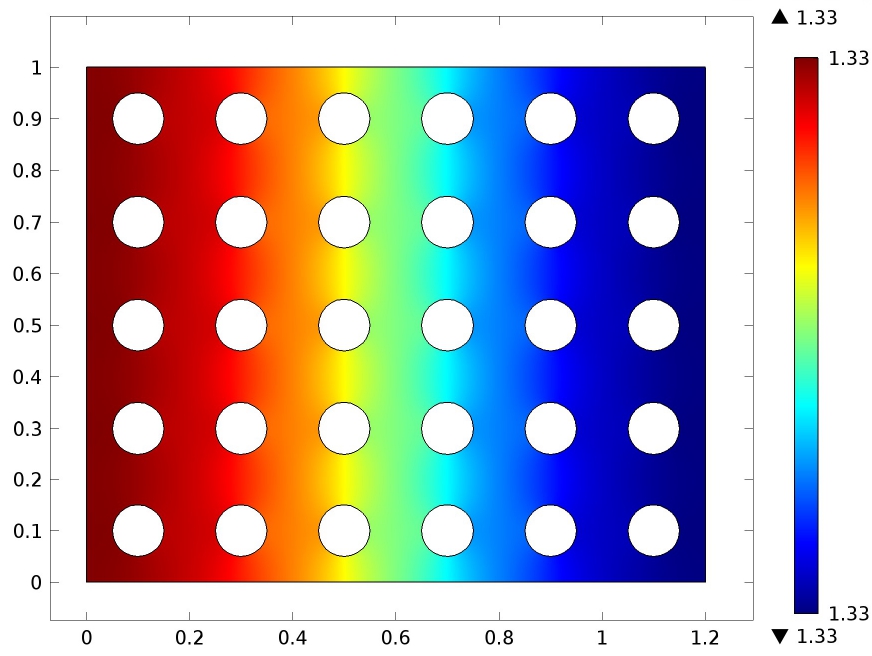} }}
  	\caption{Concentration of the second mobile species $M_2$ in $\Omega_\epsilon^p$ for different time.}
  \end{figure}
  \begin{figure}[h!]
  	\begin{center}
  		\includegraphics[width=10cm, height=6.5cm]{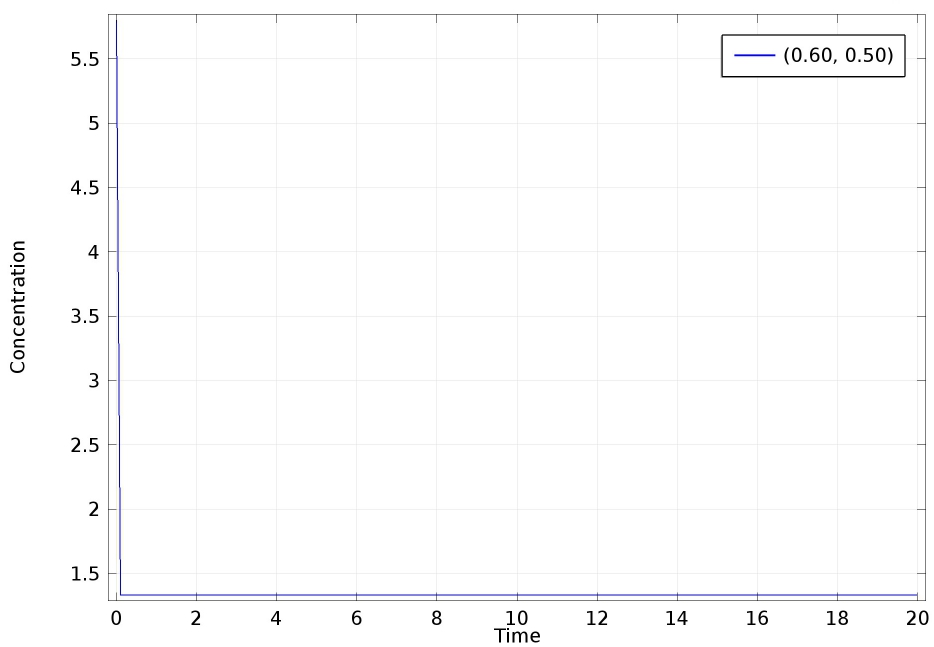}
  		\caption{Concentration of $M_2$ at the point $(0.6,0.5)$ in $\Omega_\epsilon^p$ in $20s$.}
  	\end{center}
  \end{figure}
  Similarly, the molar concentrations of $M_2$ for different time is plotted in Figure $4$ and the change of concentration can be seen in the Figure $5$. Just like $M_1$, here is also a jump in the concentration at $t=0s$ and the concentration became $5.8$. While at $t=0.1s$ it takes the value $1.33$. 
   \subsection{Solution of the Cell problems}
   In order to simulate the upscaled equations, we need to evaluate the effective diffusion tensors for the two mobile species $M_1$ and $M_2$. We commence by solving the cell problems \eqref{eqn:CP} and the solutions is shown in Figure $6$ for $j=1, 2$.
   \begin{figure}[H]
   	\centering
   	\subfloat[\centering For $j=1$]{{\includegraphics[width=7cm]{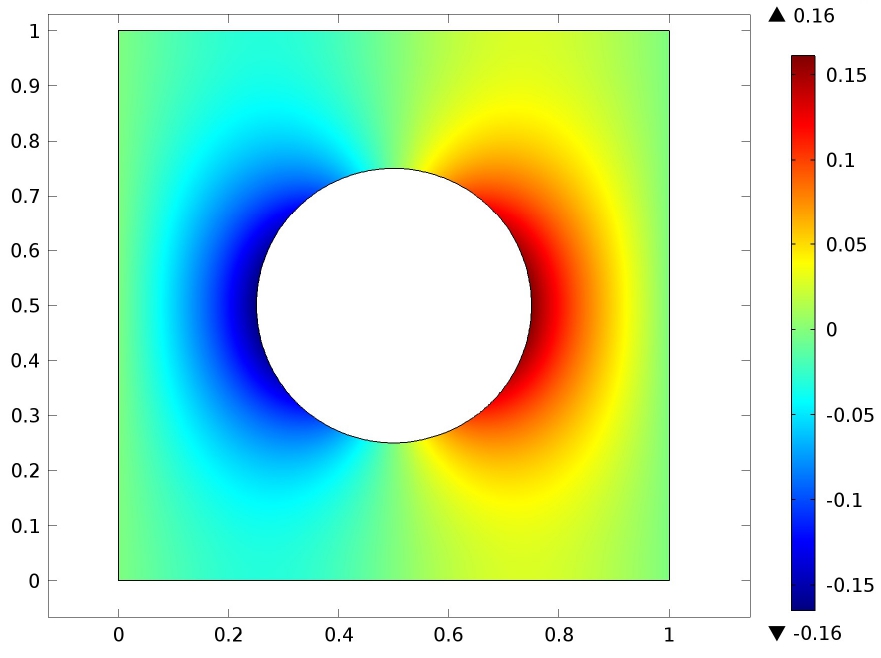} }}
   	\qquad
   	\subfloat[\centering  For $j=2$]{{\includegraphics[width=7cm]{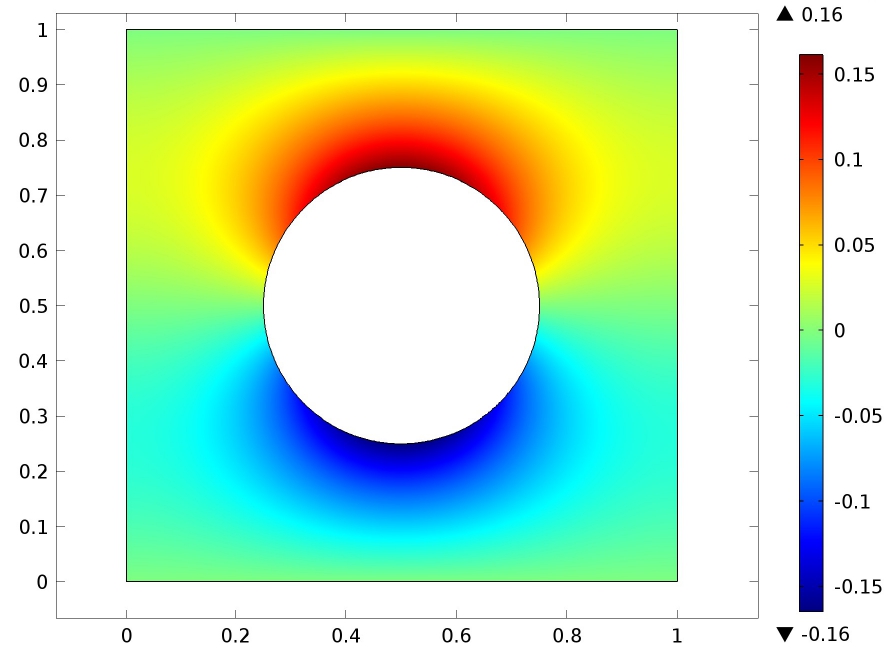} }}
   	\caption{Solution of the cell problems}
   \end{figure}
   We compute the effective tensors with the help of ``Derived values'' feature available in COMSOL. Thus we obtain
   \begin{align}
   	A=(a_{ij})_{1\le i,j\le 2}=\begin{bmatrix}
   		0.8358 & -2.91316\times 10^{-12} \\
   		-2.91334\times 10^{-12} & 0.8358
   	\end{bmatrix},\label{eqn:a}\\
   B=(b_{ij})_{1\le i,j\le 2}=\begin{bmatrix}
   	1.67161 & -5.82632\times 10^{-12} \\
   	-5.82667\times 10^{-12} & 1.67161
   \end{bmatrix}.\label{eqn:b}
   \end{align}
   \subsection{Simulation of the macromodel}
   We employ the idea of \cite{van2009crystal} for the simulation of the homogenized equations. We kept the same parameter values and the regularized parameter used for the micromodel as in \eqref{eqn:PV}. The effective homogenized matrices $A$ and $B$ are given by \eqref{eqn:a} and \eqref{eqn:b}. Again we choose the ``Normal'' mesh to discretize the domain $\Omega$ and solve the system for $t=20s$. 
   \begin{figure}[H]
   	\centering
   	\subfloat[\centering Concentration of $M_1$ for $t=5s$]{{\includegraphics[width=7cm]{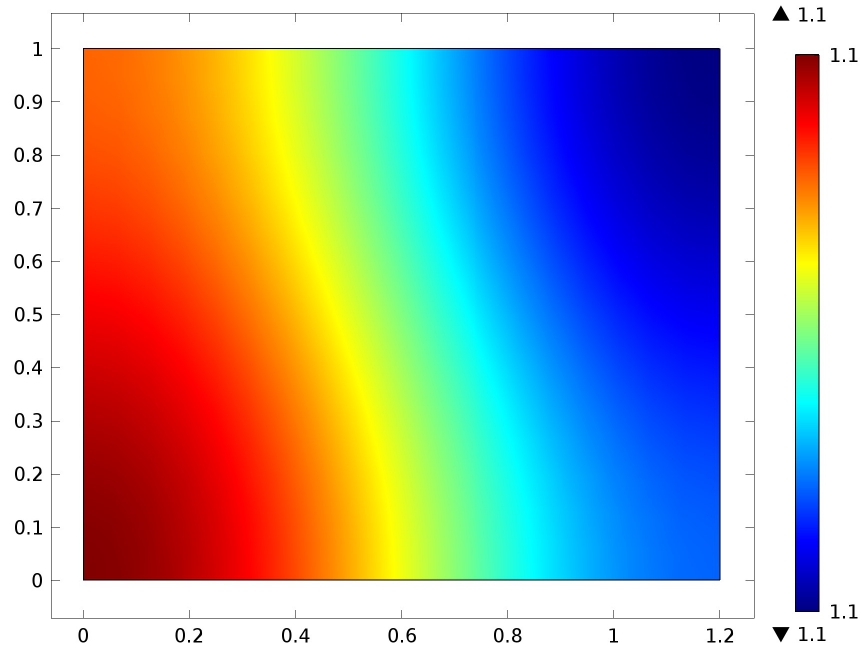} }}
   	\qquad
   	\subfloat[\centering  Concentration of $M_1$ for $t=10s$]{{\includegraphics[width=7cm]{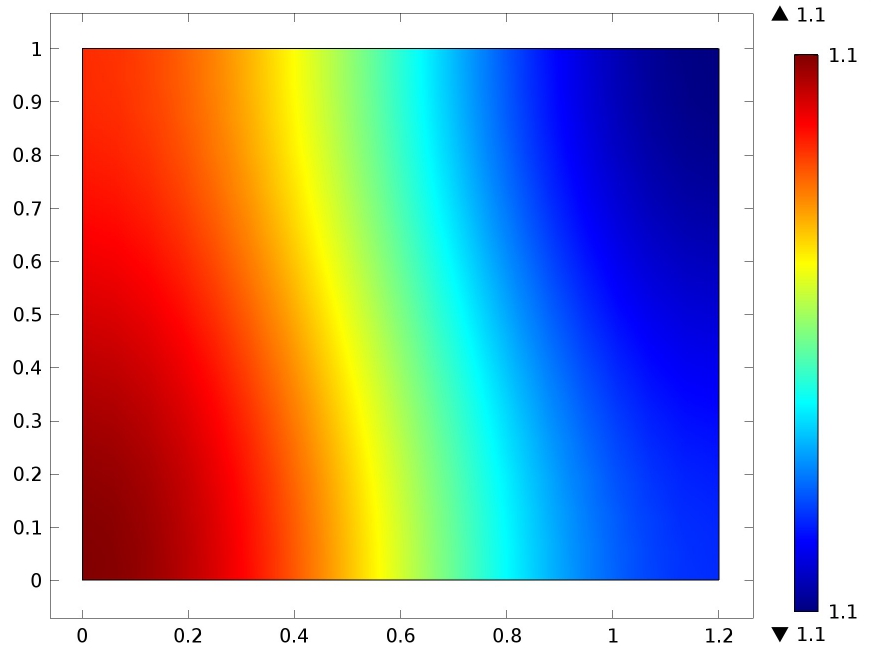} }}\\
   	\centering
   	\subfloat[\centering Concentration of $M_1$ for $t=15s$]{{\includegraphics[width=7cm]{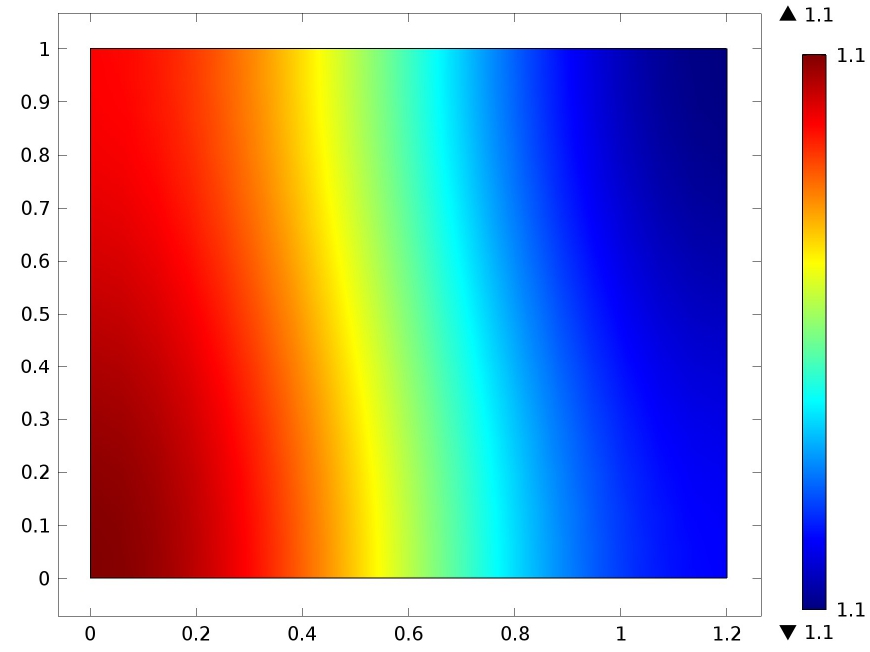} }}
   	\qquad
   	\subfloat[\centering  Concentration of $M_1$ for $t=20s$]{{\includegraphics[width=7cm]{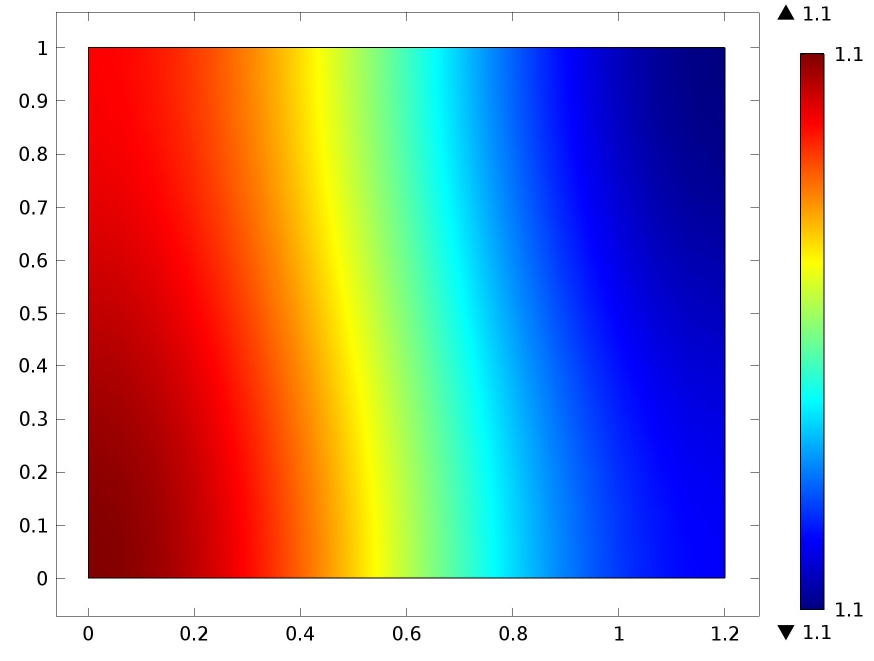} }}
   	\caption{Concentration of the first mobile species $M_1$ in $\Omega$ for different time.}
   \end{figure}
   \begin{figure}[h!]
   	\begin{center}
   		\includegraphics[width=10cm, height=6.5cm]{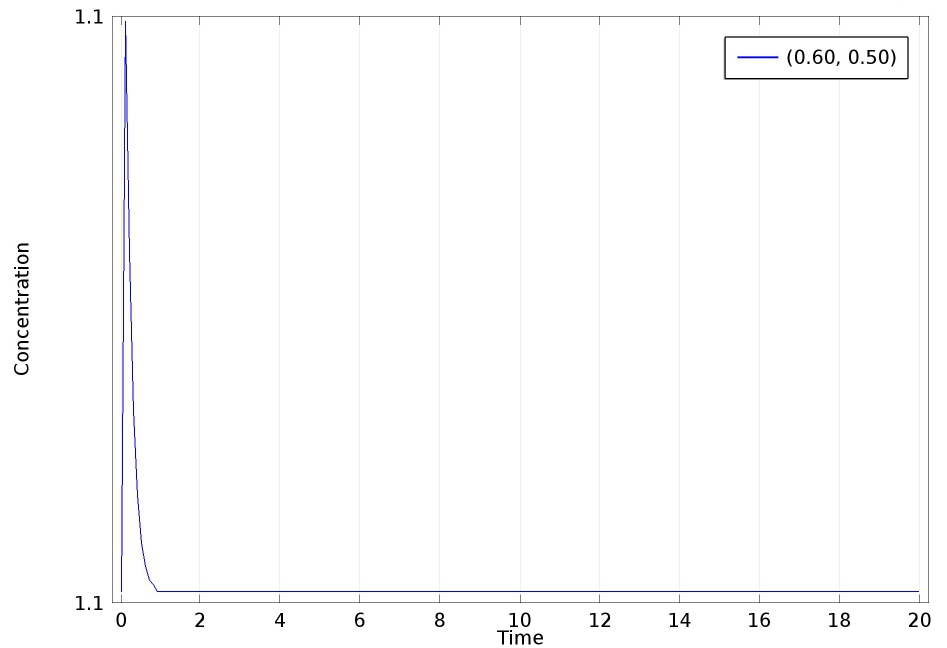}
   		\caption{Concentration of $M_1$ at the point $(0.6,0.5)$ in $\Omega$ in $20s$.}
   	\end{center}
   \end{figure}
    \begin{figure}[H]
   	\centering
   	\subfloat[\centering Concentration of $M_2$ for $t=5s$]{{\includegraphics[width=7cm]{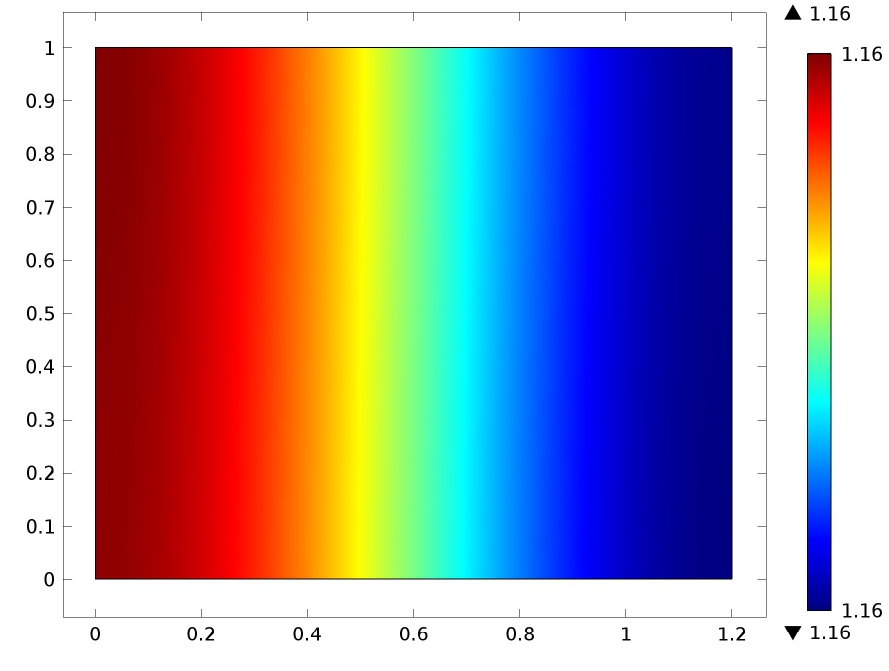} }}
   	\qquad
   	\subfloat[\centering  Concentration of $M_2$ for $t=10s$]{{\includegraphics[width=7cm]{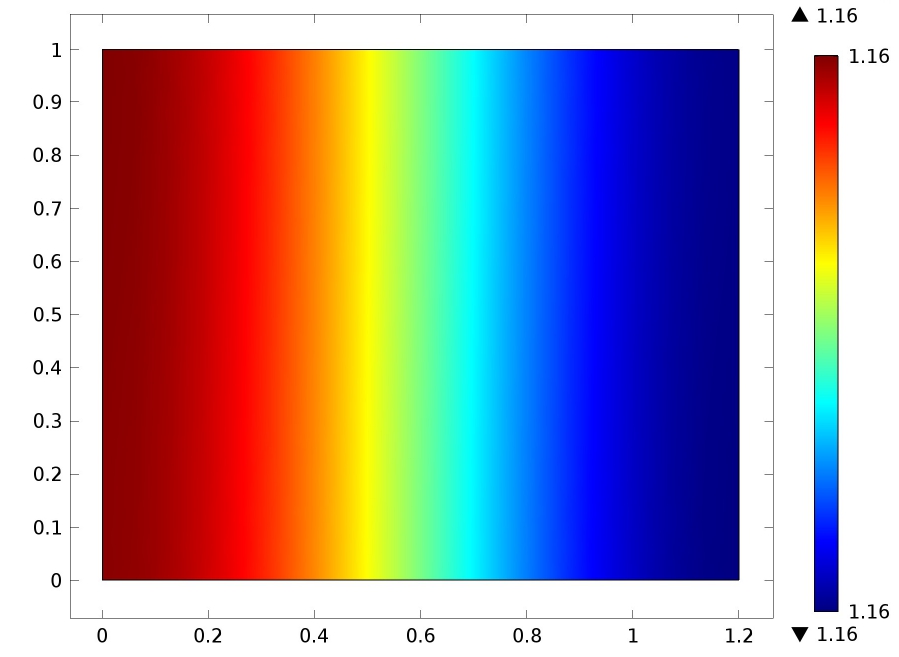} }}\\
   	\centering
   	\subfloat[\centering Concentration of $M_2$ for $t=15s$]{{\includegraphics[width=7cm]{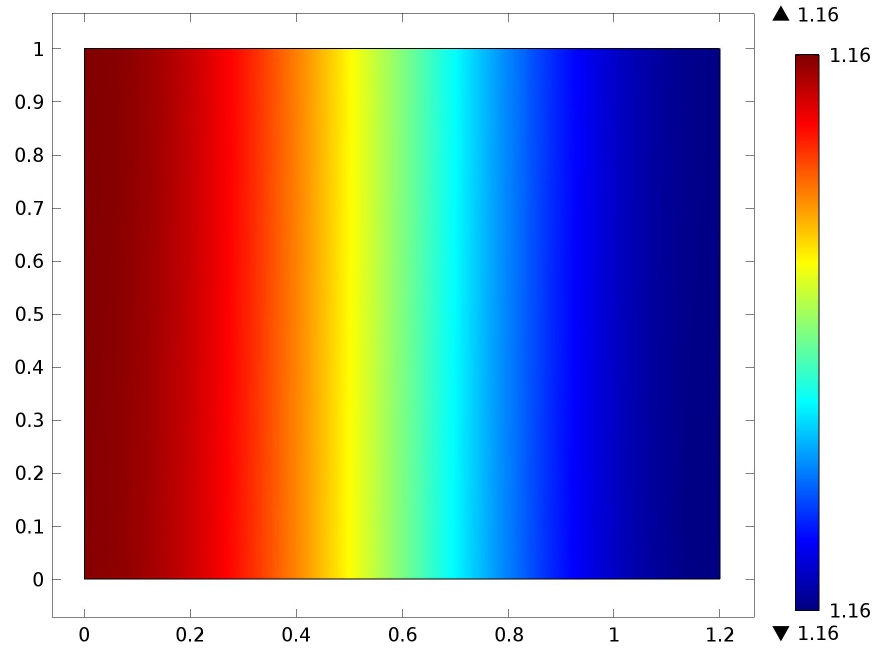} }}
   	\qquad
   	\subfloat[\centering  Concentration of $M_2$ for $t=20s$]{{\includegraphics[width=7cm]{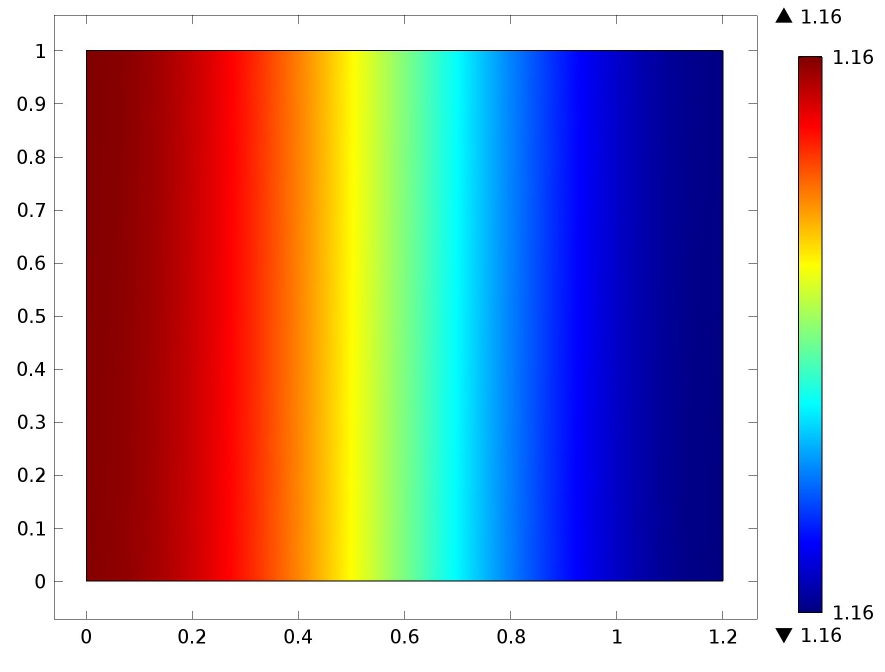} }}
   	\caption{Concentration of the second mobile species $M_2$ in $\Omega$ for different time.}
   \end{figure}
   \begin{figure}[h!]
   	\begin{center}
   		\includegraphics[width=10cm, height=6.5cm]{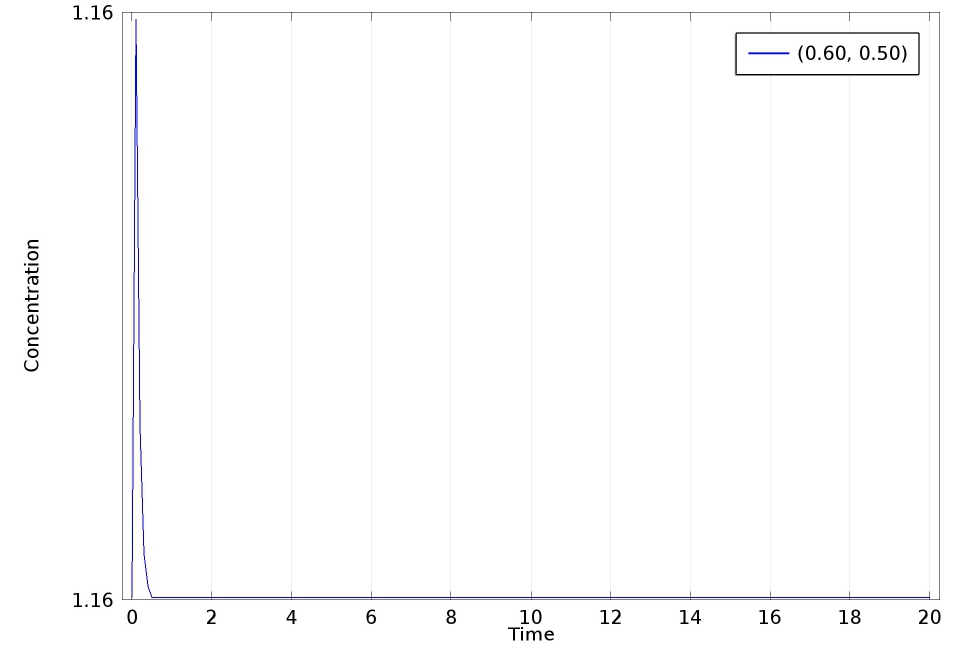}
   		\caption{Concentration of $M_2$ at the point $(0.6,0.5)$ in $\Omega$ in $20s$.}
   	\end{center}
   \end{figure}
   In this case, the time taken by the solver is $1s$. We repeat the similar computation for $M_1$ and $M_2$ in the macro model. The results are shown in Figure  $7-10$. However, the noticeable points are as follows: $(i)$ We use the same mesh to simulate the micro and macro system. Although the solver takes very little time to solve the macro problem in comparison to the micro problem. Hence the upscaled model is computationally efficient, $(ii)$ There is no such jump in concentration in the graph Figure $8$ and Figure $10$. By comparing Figure $2$ and $7$, we can conclude that the solutions of the homogenized equations agree very well with the solutions of the original micro-scale model. It can also be understandable by comparing Figure $3$ and $8$. Therefore the upscaled equations describe the behavior of the microscale model very well. Thus, homogenization proved to be an efficient tool to deal with the problems arising from the microscopically heterogeneous medium.  
  \section{Conclusion}
  We study crystal dissolution and precipitation in the context of a porous medium. The model takes care of the accumulation of the immobile species on the grain boundary. Using the homogenization technique, we derive the macroscopic model. In this article, we wish to understand the error caused by replacing a heterogeneous solution with a homogenization one together with numerical experiments. We observe that the macro model is advantageous for numerical simulations. Since it takes less time compared to the micromodel, it will reduce the computational cost for real-world applications. Furthermore, the numerical simulation for a test problem shows that the solution of the homogenized equation approximates the solution of the microscopic model very well. In this way, we validate the homogenization procedure and establish that it's an efficient tool to deal with such heterogeneous problems.
	\bibliographystyle{acm}
	\bibliography{sfw}
\end{document}